\pdfoutput=1
\RequirePackage[l2tabu, orthodox]{nag}

\documentclass[reqno]{amsart}

\usepackage{lmodern}
\usepackage[T1]{fontenc}
\usepackage[utf8]{inputenc}
\usepackage[english]{babel}
\usepackage{microtype} 

\usepackage{amsmath,amssymb,amsthm,mathrsfs,latexsym,mathtools,mathdots,booktabs,enumerate,tikz,bm,url}
\usepackage{faktor}
\usepackage{hyperref}

\usepackage[enableskew,vcentermath]{youngtab}
\usepackage[centertableaux]{ytableau}

\usetikzlibrary{arrows,positioning}
\usepackage[capitalize,noabbrev]{cleveref}

\usepackage{todonotes}
\presetkeys%
    {todonotes}%
    {inline,backgroundcolor=yellow}{}

\usepackage{graphicx}

\usepackage{enumitem}

\usepackage{ifpdf}
\graphicspath{{./figures/}{./}}
\usepackage{epstopdf}
\DeclareGraphicsExtensions{.png,.jpg,.eps,.epsf,.pdf}

\newtheorem{theorem}{Theorem}
\newtheorem{proposition}[theorem]{Proposition}
\newtheorem{lemma}[theorem]{Lemma}
\newtheorem{corollary}[theorem]{Corollary}

\theoremstyle{definition}
\newtheorem{definition}[theorem]{Definition}
\newtheorem{example}[theorem]{Example}
\newtheorem{remark}[theorem]{Remark}

\newtheorem{problem}[theorem]{Problem}

\newcommand{\Cat}{\mathbf{Cat}} 

\newcommand{\defin}[1]{\emph{#1}}

\newcommand{\setN}{\mathbb{N}}
\newcommand{\setZ}{\mathbb{Z}}

\newcommand{\avec}{\mathbf{a}}
\newcommand{\bvec}{\mathbf{b}}

\newcommand{\xvec}{\mathbf{x}}
\newcommand{\yvec}{\mathbf{y}}

\newcommand{\dvec}{\mathbf{d}}

\newcommand{\nwvec}{\mathbf{nw}}

\newcommand{\BW}{\mathrm{BW}}
\newcommand{\OBW}{\mathrm{OBW}}
\newcommand{\DP}{\mathrm{DP}}
\newcommand{\CDP}{\mathrm{CDP}}

\newcommand{\CMP}{\mathrm{CMP}}
\newcommand{\AVL}{\mathrm{AVL}}

\DeclareMathOperator{\inv}{inv}

\DeclareMathOperator{\maj}{maj}

\newcommand{\modpart}[2]{{\{#2\}_{#1}}}
\newcommand{\qbinom}{\genfrac{[}{]}{0pt}{}}
\newcommand{\multichoose}[2]{\ensuremath{\left(\kern-.3em\left(\genfrac{}{}{0pt}{}{#1}{#2}\right)\kern-.3em\right)}}

\title{The cyclic sieving phenomenon on circular Dyck paths}

\author{Per Alexandersson}
\author{Svante Linusson}
\author{Samu Potka}
\address{Dept. of Mathematics, Royal Institute of Technology, SE-100 44 Stockholm, Sweden}
\email{per.w.alexandersson@gmail.com, linusson@math.kth.se, potka@kth.se}

\date{\today}

\begin{document}

\begin{abstract}
We give a $q$-enumeration of circular Dyck paths, which is a superset of the classical Dyck paths enumerated by the Catalan numbers.
These objects have recently been studied by Alexandersson and Panova~\cite{AlexanderssonPanova2016}.
Furthermore, we show that this $q$-analogue exhibits the cyclic sieving phenomenon under a natural action of the cyclic group.
The enumeration and cyclic sieving is generalized to
Möbius paths. We also discuss properties of a generalization of cyclic sieving,
which we call \emph{subset cyclic sieving}.
Finally, we also introduce the notion of \emph{Lyndon-like cyclic sieving}
that concerns special recursive properties of combinatorial
objects exhibiting the cyclic sieving phenomenon.
\end{abstract}

\maketitle

\setcounter{tocdepth}{1}
\tableofcontents

\section{Introduction}

Unit interval graphs are in bijection with Dyck paths, and enumerated by the Catalan numbers,
see e.g.~\cite{StanleyCatalan}. Recently, a natural generalization of these graphs was
considered in \cite{AlexanderssonPanova2016,Ellzey2016} in the study of Stanley chromatic symmetric functions.
This generalization leads to an extension of Dyck paths to \emph{circular Dyck paths}, see below for a precise definition.

The number of circular Dyck paths of size $n$ is given by the formula
\begin{align}\label{eq:enumerationIntro}
(n+2)\binom{2n-1}{n-1} - 2^{2n-1},
\end{align}
 and they are in bijection with pairs of Dyck paths of size $n$ with certain constraints,
 see \texttt{A194460} in \cite{OEIS}.
Such pairs of Dyck paths have been studied in a different context, see \cite{Baur2012},
where it is mentioned that Christian Krattenthaler previously has given a proof of \cref{eq:enumerationIntro}
``via a lengthy combinatorial computation'' starting from a recursion. Circular Dyck paths are described naturally
by their \emph{area sequences}, which naturally extend the classical area sequences of Dyck paths,
see e.g.~\cite{qtCatalanBook}.

\medskip

The main results of this paper are listed below.
\begin{itemize}[leftmargin=*]
\item We prove a $q$-analogue of \cref{eq:enumerationIntro} in \cref{P:bijective_q_analog}.
This also gives the first combinatorial proof of the fact that the number of circular Dyck
paths is given by \eqref{eq:enumerationIntro}.
In \cref{sec:qCount}, we then generalize the $q$-analogue to circular Dyck paths with \emph{width} $w$,
obtaining
\begin{align}\label{eq:circularQCountGeneralIntro}
\sum_{s \in \setZ}
\sum_{j=1}^w
q^{s^2(w+2) + s(j+1)}
\left(
\qbinom{2n-1}{n-1-(w+2) s}_q
-
\qbinom{2n-1}{n+j+(w+2) s}_q
\right)
\end{align}
in \cref{prop:circularQCount}.
The $q$-analogue of \cref{eq:enumerationIntro} is the case $w = n$.

\item In \cref{thm:mainCSP}, we prove that circular Dyck paths of width $w$ together with
\eqref{eq:circularQCountGeneralIntro} exhibit the cyclic sieving phenomenon (CSP) under
a cyclic shift of the area sequence.

\item In \cref{sec:subset}, we introduce and give a few examples of a phenomenon called \emph{subset cyclic sieving},
where the values of a polynomial $f(q)$ at $n$:th roots of unity give the number of
elements in $Y\subseteq X$ fixed under a cyclic group action on $X$, and $f(q)$ is
equal to the cardinality of $Y$. 

\item In \cref{sec:mobius}, we prove a similar $q$-formula
and instance of the CSP for paths embedded in a Möbius strip.
In the process, we prove a new CSP instance
for binary words of length $n$ under a \emph{twisted cyclic shift},
with associated polynomial
\[
 \sum_{k=0}^n q^{\binom{k}{2}}\qbinom{n}{k}_q.
\]

\item In \cref{sec:lyndon}, we focus on families of CSP instances of a
special type, parametrized by
the size $n$ of the cyclic group.
We ask the associated polynomials to fulfill the relation
\[
 f_{n/m}(1) = f_n(\exp(2\pi i/m)) \text{ whenever $m|n$}.
\]
For example, this holds for the
family of polynomials in \eqref{eq:circularQCountGeneralIntro} for each fixed $w\geq 1$.
For natural reasons, we call such a sequence of CSP instances \emph{Lyndon-like},
and we provide several more examples of this type.
\end{itemize}

\medskip

Finally, we acknowledge that the \emph{On-line Encyclopedia of Integer Sequences}, \cite{OEIS},
has been of great help in this project.
This paper also benefited from experimentation with \texttt{Sage}~\cite{sage} and its
combinatorics features developed by the \texttt{Sage-Combinat} community~\cite{Sage-Combinat}.

\subsection{Brief background on the cyclic sieving phenomenon}

The \emph{cyclic sieving phenomenon} (CSP) was introduced in 2004 by Reiner, Stanton and White~\cite{ReinerStantonWhite2004}.
It generalizes Stembridge's $q = 1$ phenomenon \cite{Stembridge1994, Stembridge1994b, Stembridge1996}.
The definition consists of three ingredients: a finite set $X$, a cyclic group $C_n = \langle g \rangle$ of order $n$ acting on $X$,
and a polynomial $f(q)$ with non-negative integer coefficients satisfying $f(1) = \left|X\right|$,
for example a generating function for $X$. Let $\omega_n$ be a primitive $n$:th root of unity, for example $e^{\frac{2\pi i}{n}}$, and, as usual, let $[n] \coloneqq \{1, \dots, n\}$.

\begin{definition}
The triple $(X, C_n, f(q))$ exhibits the CSP if for every $k \in [n]$,
\[
[f(q)]_{q = \omega_n^k} = |\{x \in X : g^k \cdot x = x\}|,
\]
that is, $f(q)$ evaluated at the $k$:th power of a primitive $n$:th root of unity is the number of fixed points of $X$ under $g^k$.
\end{definition}

Reiner, Stanton and White also gave an alternative, equivalent definition of the cyclic sieving phenomenon.
The \emph{stabilizer-order} of a $C$-orbit is the size of the stabilizer group of the elements of the orbit.

\begin{proposition}[\cite{ReinerStantonWhite2004}]\label{prop:cspPolyFromGroupAction}
The triple $(X, C, f(q))$ exhibits the CSP if $a_{\ell}$ defined by
\[
f(q) \equiv \sum_{\ell = 0}^{n-1}a_{\ell}q^{\ell} \mod (q^n - 1)
\]
is the number of $C$-orbits on $X$ for which the stabilizer-order divides $\ell$.
\end{proposition}

By now, there is a multitude of CSP results. Below are some examples.
For more, see for example the survey by Sagan \cite{Sagan2011}.
For the first one, we say that $g \in C$ with $|g| = n$ acts \emph{freely} on $[N]$
if all of its cycles are of length $n$. A slight relaxation, we say that $g$ acts \emph{nearly freely} on $[N]$
if it either acts freely or if all of its cycles have length $n$ except for one singleton. The cyclic group $C$ is said to
act (nearly) freely on $[N]$ if it has a generator acting (nearly) freely on $[N]$. Finally, recall that $\binom{[N]}{k}$ and $\multichoose{[N]}{k}$ denote the sets of $k$-subsets and $k$-multisubsets of $[N]$, respectively.
\begin{theorem}[Theorem 1.1, \cite{ReinerStantonWhite2004}]
	Suppose $C$ is a cyclic group acting nearly freely on $[N]$. Then
	\[
	\left(
	\multichoose{[N]}{k}
	, C, \qbinom{N + k - 1}{k}_q\right)
	\text{ and }
	\left(\binom{[N]}{k}, C, \qbinom{N}{k}_q\right)\]
exhibit the CSP.
\end{theorem}

Rhoades proved a CSP result for rectangular standard Young tableaux \cite{Rhoades2010}.
For a more geometric version, for example in terms non-crossing matchings in the two-row case,
see the work by Petersen, Pylyavskyy and Rhoades \cite{PetersenPylyavskyyRhoades2008}.
\begin{theorem}[Theorem 1.3, \cite{Rhoades2010}]\label{thm:rhoades}
	If $\lambda = (n^m)$, then
	\[(\mathrm{SYT}(\lambda), \langle \partial \rangle, f^{\lambda}(q))\]
	exhibits the CSP, where $\mathrm{SYT}(\lambda)$ is the set of
	standard Young tableaux of the shape $\lambda$, $\langle \partial \rangle$ is the cyclic group generated
	by the jeu-de-taquin promotion operator, and $f^{\lambda}(q)$ is the natural $q$-analogue of the hook-length formula.
\end{theorem}
In the two-row case $\lambda = (n, n)$, note that there is a bijection between $\mathrm{SYT}(\lambda)$ and Dyck paths of size $n$.
Small Schröder paths also exhibit the CSP, see \cite{Pechenik2014}.

Another result specializing to lattice paths is the following. The \emph{major index} of a word $w$ of length $n$ is the sum of the indices $i \in [n - 1]$ such that
$w_i > w_{i+1}$. A pair $(i, j)$ is an \emph{inversion} of $w$ if $i < j$ but $w_i > w_j$, and $\inv(w)$ is the \emph{number of inversions} in $w$.
\begin{theorem}[A reformulation of Proposition 4.4, \cite{ReinerStantonWhite2004}]\label{thm:wordcsp}

Let $X_n(\mu)$ be the set of words of length $n$ and content $\mu,$ that is, each word has $\mu_i$ entries equal to $i$.
Let the cyclic group $C_n$ act on $X$ by cyclic shift, and let
\[
f_n(\mu; q) \coloneqq \qbinom{n}{\mu}_q = \sum_{w \in X_n(\mu)} q^{\maj(w)} = \sum_{w \in X_n(\mu)} q^{\inv(w)}.
\]
Then $(X_n,C_n,f_n(\mu;q))$ exhibits the CSP.
\end{theorem}
In the case $\mu = (\mu_1, \mu_2)$ with $\mu_1 + \mu_2 = n$, $X_n(\mu)$ is in an
obvious bijection with, for example, lattice paths starting at $(0, 0)$ and with steps from $\{(0, 1), (1, 0)\}$.
See \cite{AhlbachSwanson2018} for a refinement of \cref{thm:wordcsp}.

\subsection{Background on \texorpdfstring{$q$}{q}-analogues}

In the previous examples we saw how \emph{$q$-analogues} appear in the context of the cyclic sieving phenomenon.
We will also encounter them in this paper and hence introduce them here.
The starting point is the definition $[n]_q \coloneqq \frac{1-q^n}{1-q} = 1 + q + \dotsb + q^{n-1}$,
which is motivated by the observation
\[
\lim_{q \rightarrow 1} \frac{1-q^n}{1-q} = n.
\]
 Then it is natural to define the \emph{$q$-factorial}
\begin{align*} [n]_q! \coloneqq &~[1]_q \cdot [2]_q \cdots [n-1]_q \cdot [n]_q = \frac{1-q}{1-q}\cdot \frac{1-q^2}{1-q} \cdots \frac{1-q^{n-1}}{1-q}\cdot \frac{1-q^n}{1-q}\\ =&~1 \cdot (1+q)\cdots(1+q + \dots + q^{n-2})\cdot(1+q + \dots + q^{n-1}).
\end{align*}
While $n!$ counts the number of permutations on $[n]$, it is well-known (see, for example, \cite{StanleyEC1})
that $[n]_q! = \sum_{\sigma \in S_n} q^{\mathrm{inv}(\sigma)}$

Having defined $q$-factorials,
the next natural step is to define \emph{$q$-binomial coefficients}
(also called Gaussian binomial coefficients, Gaussian coefficients and Gaussian polynomials) by
\begin{align*}
\qbinom{n}{k}_q \coloneqq &~\frac{[n]_q!}{[n-k]_q![k]_q!}
\text{ for } 0 \leq k \leq n,
\end{align*}
and letting $\qbinom{n}{k}_q \coloneqq 0$ otherwise.
One combinatorial interpretation of the $q$-binomial coefficient is that it counts
the number of $k$-dimensional subspaces of the $n$-dimensional vector
space over the $q$-element field, see \cite{StanleyEC1} for the details.

Many identities for binomial coefficients have their counterparts for $q$-binomial coefficients.
For example, we have the symmetry
\[
\qbinom{n}{k}_q = \qbinom{n}{n-k}_q
\]
and the \emph{$q$-Pascal identities}
\[
\qbinom{n}{k}_q = q^k\qbinom{n-1}{k}_q + \qbinom{n-1}{k-1}_q \text{ and }
\qbinom{n}{k}_q = \qbinom{n-1}{k}_q + q^{n-k}\qbinom{n-1}{k-1}_q.
\]

A useful tool for proving CSP results is the $q$-Lucas theorem below.
We shall make use of the following notation.
Given $a \in \setZ$ and $d \in \setN$, let $\{ a \}_d$
denote the remainder of $a$ mod $d$, so that
$a  = d\lfloor a/d \rfloor + \{ a \}_d$.
\begin{lemma}[See, for example, \cite{Sagan1992}]\label{lem:qLucas}
We have that
\[
 \qbinom{n}{k}_q \equiv \binom{ \lfloor n/d \rfloor }{ \lfloor k/d \rfloor }
 \qbinom{ \{ n \}_d }{  \{ k \}_d }_q\ \mod \Phi_d,
\]
where $\Phi_d$ is the $d$:th cyclotomic polynomial.
\end{lemma}
In particular, \cref{lem:qLucas} implies that with $q$ a primitive $d$:th root of unity,
\[
\qbinom{n}{k}_q = \binom{ \lfloor n/d \rfloor }{ \lfloor k/d \rfloor }
 \qbinom{ \{ n \}_d }{  \{ k \}_d }_q,
\]
a fact we use extensively in later sections.

The following two well-known results due to MacMahon are also related to our work.
\begin{lemma}[See Theorem 3.7 in \cite{Andrews1976}]\label{lem:easyWordQAnalogue}
Let $\BW(k, m)$ denote the set of binary words of length $k + m$ with $k$ $1$s.
Then
\[
\sum_{\bvec \in \BW(k,m)} q^{\maj(\bvec)} = \qbinom{k + m}{k}_q.
\]
\end{lemma}
\begin{proposition}[See e.g.\ the lemma on p.\ 255 in \cite{FurlingerHofbauer1985}]
The major index of binary words corresponding to Dyck paths generates the classical Carlitz $q$-analogue of the Catalan numbers:
\[
\sum_{D \in \DP(n)} q^{\maj(D)} = \frac{1}{[n+1]_q} \qbinom{2n}{n}_q.
\]
\end{proposition}
When evaluating the right hand side at $e^{2\pi i k/n}$, we obtain non-negative integers which count fixed points under promotion.
That action can be described either by bijecting to $2\times n$ SYT, or by considering $2\pi / n$ rotations of perfect matchings
in a $2n$-gon. The special case of \cref{thm:rhoades} mentioned in the previous section is a refinement of this, using rotations
of $\pi / n$ instead.

\section{Enumeration of circular Dyck paths}

A \emph{Dyck path} may be described via its \emph{area sequence}. For example,
the path $(0,1,2,3,2,2)$ corresponds to
\begin{align}\label{eq:dyckPathExample}
\begin{ytableau}
*(lightgray)   & *(lightgray)   & *(lightgray)   &      &    & *(yellow) 6 \\
*(lightgray)   & *(lightgray)   &     &     & *(yellow) 5\\
     &    &     & *(yellow) 4\\
     &     & *(yellow) 3\\
    & *(yellow) 2\\
*(yellow) 1
\end{ytableau}
\end{align}
where the area sequence specifies the number of white squares in each row, from bottom to top.
The number of Dyck paths of size $n$ is given by the $n$:th Catalan number, $\frac{1}{n+1}\binom{2n}{n}$.

\medskip
A \emph{circular Dyck path} of size $n$ is specified via
an area sequence such that $a_1,\dotsc,a_n$ satisfy
\begin{itemize}
 \item $0 \leq a_i \leq n-1 $ for $1 \leq i \leq n$,
 \item $a_{i+1} \leq a_{i} + 1$ for $1 \leq i \leq n$,
\end{itemize}
where the index is taken mod $n$ in the second condition.
This set is denoted $\CDP(n)$.
The subset of paths with $a_1=0$ correspond to classical Dyck paths, $\DP(n)$.
Circular Dyck paths can also be illustrated as diagrams.
\begin{example}\label{ex:circularDyckPath}
For example, $\avec = (3,4,2,3,2,3)$ is illustrated as
\[
\begin{ytableau}
\none    &\none &\none &\none &\none &\none &\none & *(lightgray) & *(lightgray) &   &   &   & *(yellow) 1 \\
\none[3] &\none &\none &\none &\none &\none & *(lightgray) & *(lightgray) &   &   &   & *(yellow) 6 \\
\none[2] &\none &\none &\none &\none & *(lightgray) & *(lightgray) & *(lightgray)  &   &   & *(yellow) 5 \\
\none[3] &\none &\none &\none & *(lightgray) & *(lightgray) &   &   &  & *(yellow) 4 \\
\none[2] &\none &\none & *(lightgray) & *(lightgray) & *(lightgray) &   &   & *(yellow) 3 \\
\none[4] &\none & *(lightgray) &  &   &   &   & *(yellow) 2 \\
\none[3] & *(lightgray) & *(lightgray) &   &   &   & *(yellow) 1
\end{ytableau}
\]
where the top row is a repetition of the first row to illustrate the cyclic nature of the graph.
\end{example}

It is often convenient to describe circular Dyck paths as paths along the border of the white squares,
see \cref{fig:bij-with-paths}.
For this to be uniquely defined, one has to fix a starting point $\xvec=(x_0,0), 1\le x_0\le n$.
We denote such a path by $(\xvec, \bvec)$,
where $\bvec=(b_1,\dots,b_{2n})\in\{0,1\}^{2n}$ is a binary sequence with $n$ $0$s and $1$s, respectively.
Here $0$ corresponds to a right step and $1$ to an upstep, with
$b_{2n}=1$, that is, the last step is up.

\begin{figure}[!ht]
\centering
\includegraphics[width=0.5\textwidth]{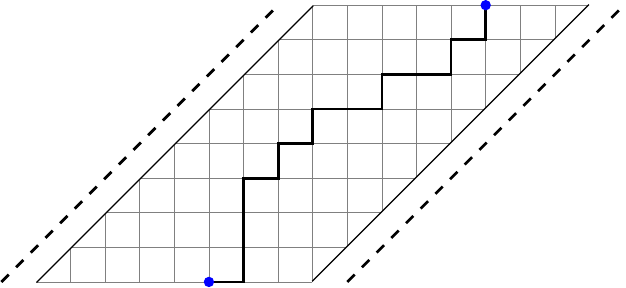}
\caption{Example of the bijection with lattice paths, with $n=8$.
The path in the middle, starting at $(6,0)$
and with binary word \texttt{0111010100100101},
corresponds to the area sequence
$\avec=(2,3,4,4,4,3,2,2)$.
The area sequence is determined by the number of (whole) squares to the right of the path in each row.
The two dashed diagonals are never touched by a lattice path corresponding to
an area sequence.
}\label{fig:bij-with-paths}
\end{figure}


The condition $b_{2n}=1$ is needed to make the starting point well defined.
Note that one may not have $x_0=n+1$, since the last step is up,
it would mean passing the illegal point $(2n+1,n-1)$.

We also study circular Dyck paths with a width different from the height,
which are defined in an analogous manner.
A \emph{circular Dyck path} of height $n$ and width $w$ is specified via
an area sequence such that $a_1,\dotsc,a_n$ satisfy
\begin{itemize}
 \item $0 \leq a_i \leq w-1 $ for $1 \leq i \leq n$,
 \item $a_{i+1} \leq a_{i} + 1$ for $1 \leq i \leq n$,
\end{itemize}
where the index is taken mod $n$ in the second condition.
This set is denoted $\CDP(n,w)$. Equivalently we can think of the elements in  $\CDP(n,w)$ as $(\xvec,b)$,
where $\xvec=(x_0,0)$, $1\le x_0\le w$, is the starting point and $\bvec=(b_1,\dots,b_{n+w})\in\{0,1\}^{n+w}$ is a binary sequence with
$n$ $1$s and $w$ $0$s such that the corresponding path stays between the diagonals $y=x$ and $y=x-(w+2)$,
and $b_{n+w}=1$.

There is a natural $C_n$-action on $\CDP(n,w)$, where the generator shifts the area sequence is cyclically by one step to the right.
We let $\alpha$ denote such a cyclic shift.

\subsection{Bijection with tuples of Dyck paths}

There is a bijection between circular Dyck paths and pairs of Dyck paths with certain peak conditions.
\emph{Peaks} in the path $P$ below are occurrences of east-steps followed by a north-step. The \emph{height}
of the first peak of $P$ is the number of east-steps before the first north-step. The height of
the last peak of $P$ is the number of north-steps at the end of $P$. Exchanging north and east-steps in
these definitions gives the corresponding definitions for $Q$.

For a Dyck path $Q$, let $h_f(Q)$ and $h_lQ)$ be the heights of the first peak and the last peak of $Q$.

\begin{lemma}[{\cite[Lemma 5]{AlexanderssonPanova2016}}]\label{lem:pair-bijection}
Circular Dyck paths of size $n$ are in bijection with pairs $(P,Q)$ of ordinary Dyck paths of size $n$,
such that
\[
h_f(P)+h_l(Q)\ge n \qquad \text{ and } \qquad h_l(P)+h_f(Q)\ge n.
\]
\end{lemma}

The bijection in the previous lemma is illustrated in \cref{fig:asPairsBijection}.
\begin{figure}[!ht]
\centering
\includegraphics[width=0.5\textwidth]{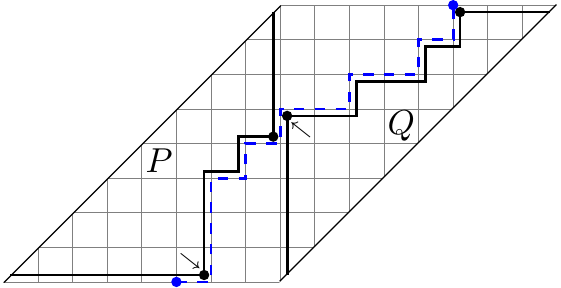}
\caption{The bijection with pairs of Dyck paths.
The circular area sequence of the circular Dyck path in the figure is $(2,3,4,4,4,3,2,2)$.
The first and last peaks have been marked, and arrows point the first peak in each path.
}\label{fig:asPairsBijection}
\end{figure}

\begin{lemma}
The set $\CDP(kn,n)$ is in bijection with $(2k)$-tuples of Dyck paths, $(P_1,\dotsc,P_{2k})$,
such that
\[
h_l(P_j)+h_f(P_{j+1})\ge n, \text{ for } 1\leq j <2k \quad \text{ and } \quad h_l(P_{2k})+h_f(P_{1})\ge n.
\]
\end{lemma}
\begin{proof}
This follows from simply extending the idea in \cref{lem:pair-bijection},
as shown in \cref{fig:asTupleBij}.
\begin{figure}[!ht]
\centering
\includegraphics[width=0.5\textwidth]{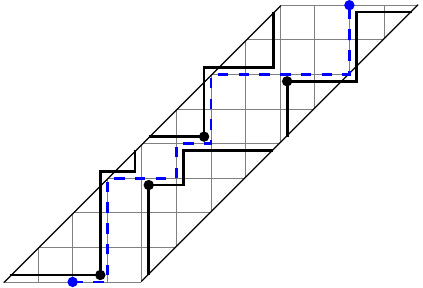}
\caption{Here we have a path (dashed) from $\CDP(8,4)$,
which is mapped to a $4$-tuple of Dyck paths of size $4$ (solid).
The first peak has been marked in each Dyck path.
}\label{fig:asTupleBij}
\end{figure}
\end{proof}

\subsection{Circular Möbius paths}

It is natural to ask what happens if we have a \emph{single} Dyck path $P$ of size $n$, such that $(P,P)$
satisfies the peak condition in \cref{lem:pair-bijection}.
It is straightforward to show that
such Dyck paths are in bijection with $(\xvec,\bvec)$ in $\CDP(n)$ such that
$\bvec = (b_1,\dotsc,b_{2n})$ fulfills the relation
\begin{equation}\label{eq:mobiusCondition}
 b_i  = 1-b_{n+i}\quad \text{ for all } \quad i \in [n].
\end{equation}
Note that, in particular, $b_n=0$ as we always have $b_{2n}=1$.
The starting point $\xvec=(x_0,0)$ is uniquely determined by $\bvec$ since $b_n=0$ must correspond to an east step that
ends on the vertical line $x=n+1$.
From this it is easy to see that all possible $\bvec$ correspond to exactly one path.
We let $\CMP(n)\subseteq \CDP(n)$ denote this set,
and refer to such paths as \defin{circular Möbius paths},
see \cref{fig:mobiusExample} for an example.

\begin{figure}[!ht]
\centering
\includegraphics[width=0.45\textwidth]{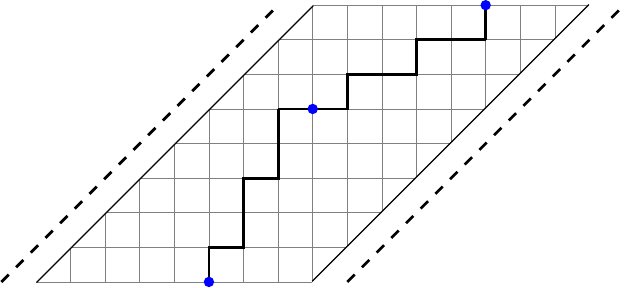}
\includegraphics[width=0.45\textwidth]{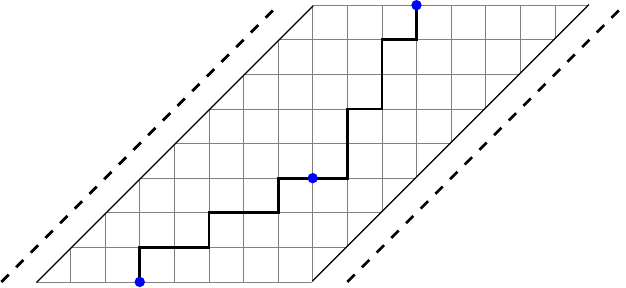}
\caption{Two Möbius paths in $\CMP(8)$.
Note that each path is the concatenation of two smaller paths, the second part being a reflection
of the first. Hence the ``Möbius'' name.
}\label{fig:mobiusExample}
\end{figure}

\begin{lemma}\label{lem:mobiusBijection}
We have that $|\CMP(n)| = 2^{n-1}$.
\end{lemma}
\begin{proof}
As noted above the path in $\CMP(n)$ is determined uniquely by the first $n-1$ steps in $\bvec$. 
\end{proof}

\subsection{A \texorpdfstring{$q$}{q}-analogue}

As mentioned, circular Dyck paths and NE-lattice paths in general correspond to pairs
$(\xvec, \bvec)$ of a starting point $\xvec$ and a binary word $\bvec$ where $b_i = 1$
if the $i$:th step is a north step, and $b_i = 0$ otherwise.
For example, the path in \cref{fig:asPairsBijection} gives the word $0100100101011101$.

The \emph{major index} of a circular Dyck path $\avec$ (corresponding to $(\xvec, \bvec)$) is
defined as the major index of the binary word $\bvec$.
Recall that the major index of a binary word $w$ of length $n$ is the sum of all $i\in[n-1]$ such that $w_i>w_{i+1}$.

Define the following $q$-analogue of circular Dyck paths:
\begin{equation}\label{eq:cdpQanalogueDef}
|\CDP(n,w)|_q \coloneqq  \sum_{(\xvec, \bvec) \in \CDP(n,w)} q^{\maj(\bvec)}.
\end{equation}

\medskip

We end this section by proving a $q$-analogue of \cref{eq:enumerationIntro} in the introduction,
that enumerates $\CDP(n)$. The proof mimics the ideas of \cite[p. 255]{FurlingerHofbauer1985},
In the next section, we extend the method and do the same for $\CDP(n,w)$.

We begin the proof with a lemma generalizing $q$-ballot numbers.

\begin{lemma}\label{L:gen_q_ballot} For any $n\ge 1$,
summing over all NE-paths $\bvec$ starting in $(x,0)$, ending in $(i,j)$, and never touching the $x=y$ diagonal,
and with $i\ge j, x\ge 0$, we get
 \begin{equation}\label{eq:gen_q_ballot}
 \sum_{\bvec} q^{\maj(\bvec)}= \qbinom{i+j-x}{j}_q
- q^{x} \qbinom{i+j-x}{j-x}_q.
\end{equation}
\end{lemma}
\begin{proof}
We proceed by induction over $x$. For $x=0$ it is clearly true since then there are no paths.
The maj-count of all NE-paths from $(x,0)$ to $(i,j)$ is
$\qbinom{i+j-x}{j}_q$, and we will now count and subtract the paths that touch the diagonal $x=y$. The idea for this proof comes from \cite[p. 255]{FurlingerHofbauer1985},
where they construct a major-index preserving bijection between sets of lattice paths.

In this proof we define the depth of a path to measure how far beyond the diagonal $x=y$ the path goes, or, more formally,
to be the largest value of $s-r$ for any point $(r,s)$ on the path. We now define a bijection $\varphi$ that maps a path
$\bvec$ ending in $(i,j)$ with depth $d\ge 0$ to a path ending in $(i+1,j-1)$ with depth $d-1$.

Let $(r,s)$ be the first point of maximal depth on the path. Since $x>0$, $(r,s)$ is not the starting point.
If it is the point directly after the starting point, we must have $r=s=1$.
Otherwise, the last two steps reaching $(r,s)$ are north steps.
The map $\varphi $ is defined by switching the north step just before $(r,s)$ to an
east step, that is,  relabeling $b_{r+s-x}$ from $1$ to $0$.
In $\varphi(\bvec)$  the position $(r,s-1)$ will then be the last point of maximal depth and it is
thus easy to find it and define $\varphi^{-1}$.
The corner in position $(r,s)$ has been replaced with a corner in $(r,s-1)$,
unless $r=s=1$. In any case, $\maj(\bvec)=\maj(\varphi(\bvec))+1$.
Thus $\varphi$ is a bijection between the paths from $(x,0)$ to $(i,j)$ that touch the diagonal $x=y$ and
paths from $(x,0)$ to $(i+1,j-1)$ that touch the diagonal $x=y+1$, with a shift of $q^1$.
By induction, the maj-count of paths from $(x,0)$ to $(i+1,j-1)$ that do not touch the diagonal $x=y+1$ is
\[
\qbinom{i+j-1-(x-1)}{j-1}_q - q^{x-1} \qbinom{i+j-1-(x-1)}{j-1-(x-1)}_q,
\]
and thus the maj-count of those touching the diagonal $x=y+1$ is
\[
q^{x-1} \qbinom{i+j-1-(x-1)}{j-1-(x-1)}_q=q^{x-1} \qbinom{i+j-x}{j-x}_q.
\]
Using $\varphi$ we thus get that the maj-count of the paths from $(x,0)$ to $(i,j)$ that
touch the diagonal $x=y$ is $q^{x} \qbinom{i+j-x}{j-x}_q$,
which gives the formula claimed in the lemma.
\end{proof}

\begin{proposition}\label{P:bijective_q_analog} For any $n\ge 1$,
\begin{align}\label{eq:circularQCount}
 |\CDP(n)|_q = n \qbinom{2n-1}{n-1}_q
- \sum_{j=1}^{n} q^{j} \qbinom{2n-1}{n+j}_q
- \sum_{j=1}^{n} \qbinom{2n-1}{j-2}_q.
\end{align}
\end{proposition}

\begin{proof}
For each possible starting point $(x,0)$, $1\le x\le n$ the maj-count of all paths to $(x+n,n-1)$
(remember that the last step of $\bvec$ is a north step) is $\qbinom{2n-1}{n-1}_q$.
This gives the first term.
We will now subtract the paths that touches the surrounding diagonals.
Note that no path can touch both diagonals.
By \cref{L:gen_q_ballot} the maj-count of paths touching the
diagonal $x=y$ is $\sum_{x=1}^{n-1} q^{j} \qbinom{2n-1}{n-1-x}_q$, which gives the first sum.

For paths touching the diagonal $x=y+n+2$ we can use the bijection defined dually to $\varphi$ in
the proof of \cref{L:gen_q_ballot}. That is, we change an east step to a north step for the
first corner being diagonally furthest to the right. This time there clearly is no
shift in the $\maj$ of the path and we get the second sum.
\end{proof}

\section{A formula for the \texorpdfstring{$q$}{q}-analogue for circular Dyck paths}\label{sec:qCount}

The goal of this section is to express $|\CDP(n,w)|_q$ as a sum of $q$-binomial coefficients.
To achieve this, we need to consider the major index generating function
for arbitrary north-east lattice paths starting at the origin with some constraints which will be used in an inclusion-exclusion argument.

\medskip

A \defin{diagonal} is a set of lattice points of the form
$\{ \xvec + k (1,1) : k \in \setZ \}$ for some $\xvec \in \setZ^2$.
It is clear that a diagonal is uniquely specified by any point on the diagonal.
For a lattice point $\yvec$ in the non-negative quadrant, let
\begin{equation}\label{eq:hPathGenFunc}
 H(\yvec; d_1,d_2,\dotsc,d_\ell),\qquad d_i \in \setZ
\end{equation}
denote the $q$-enumeration (using major index) of north-east lattice paths $L$ from $(0,0)$ to $\yvec$,
such that $L$ includes points from each of the $\ell$ diagonals
specified by the points
\begin{equation*}
d_i(1,0)\qquad i = 1,2,\dotsc,\ell,
\end{equation*}
in this order. In other words, there must be points $p_1,\dotsc,p_\ell$
on $L$ visited in this order, such that $p_i$ is on the diagonal specified by $d_i$.
Note that by definition
\[
 H(\yvec; 0,d_1,d_2,\dotsc,d_\ell) = H(\yvec;d_1,d_2,\dotsc,d_\ell)
\]
since the starting point $(0,0)$ is on the diagonal specified by $0$. Abusing notation, we henceforth let the diagonal $d_i$ be
the unique diagonal specified by $d_i$.

We say that the configuration $(\yvec;d_1,\dotsc,d_\ell)$
is \emph{alternating} if any of the four conditions below is fulfilled:
\begin{enumerate}
 \item $0 \geq d_1 < d_2 > d_3 < d_4 > \dotsb > d_\ell$ and $\yvec$ is to the right of diagonal $d_\ell$,
 \item $0 \geq d_1 < d_2 > d_3 < d_4 > \dotsb < d_\ell$ and $\yvec$ is to the left of diagonal $d_\ell$,
 \item $0 \leq d_1 > d_2 < d_3 > d_4 < \dotsb < d_\ell$ and $\yvec$ is to the left of diagonal $d_\ell$ or
 \item $0 \leq d_1 > d_2 < d_3 > d_4 < \dotsb > d_\ell$ and $\yvec$ is to the right of diagonal $d_\ell$.
\end{enumerate}
By convention, if $\ell=0$, the configuration is considered to be alternating as well. Note that
$H(\yvec; d_1) = H(\yvec)$ if $(0, 0)$ and $\yvec$ are on different sides of $d_1$. Note also more
generally that given a non-alternating subsequence $d_k > d_{k+1} > d_{k+2}$ or $d_k < d_{k+1} < d_{k+2}$
of diagonals, $H(\yvec; \dvec) = H(\yvec; \dvec')$ where $\dvec'$ denotes $\dvec$ with $d_{k+1}$ removed.

Let $\nwvec$ denote the vector $(-1,1)$ and recall that we
identify east steps with $0$ and north steps with $1$.

\begin{lemma}\label{lem:diagonalHitMajorIndexCount}
Suppose that $(\yvec;d_1,\dotsc,d_\ell)$ is alternating and $\ell \geq 1$.
Then the generating function $H(\yvec; d_1,\dotsc,d_\ell)$ is equal to
\begin{align}\label{eq:diagonalHitRecursion}
\begin{cases}
H(\yvec - \nwvec; d_1+1,d_2+2,d_3+2,\dotsc,d_\ell+2) \times q \text{ if $d_1<0$}\\
H(\yvec + \nwvec; d_1-1,d_2-2,d_3-2,\dotsc,d_\ell-2) \phantom{\times q} \text{  otherwise.}\\
\end{cases}
\end{align}
Furthermore, both new configurations above are alternating as well.
\end{lemma}
\begin{proof}

The proof uses a similar map as the proof of \cref{L:gen_q_ballot}.
We have two different cases to consider: $d_1 < 0$ and $d_1 > 0$. Let us start with the former.
In all cases, we let $(r, s)$ be the first point maximizing the \emph{depth} $\textrm{sgn}(d_1)(x - y)$
among points $(x, y) \in L$ before $L$ meets the diagonal $d_2$. In other words, $(r, s)$ is the first
point furthest away from $d_1$ on the side opposite of $(0, 0)$, or on $d_1$ if $L$ does not cross $d_1$.

\textbf{Case $d_1 < 0$:}
Suppose $L$ is a path counted by $H(\yvec; d_1,\dotsc,d_\ell)$. A north step $1$
has to precede $(r, s)$, while an east step $0$ has to follow  it.

Let $\phi$ be the map replacing the north-step $1$ preceding $(r, s)$ with an east-step.
Note that $\phi$ is similar to $\varphi$ in the proof of \cref{L:gen_q_ballot} but
now only the points of $L$ before it meets $d_2$ are considered.

Now, $\maj(\phi(L)) = \maj(L) - 1$, the new endpoint is given by $\yvec-\nwvec$
and $\phi(L)$ hits the shifted diagonals $d_1+1$, $d_2+2,d_3+2,\dotsc,d_\ell+2$,
see \cref{fig:phiMap}.
	\begin{figure}[!ht]
		\centering
		\includegraphics[width=0.5\textwidth]{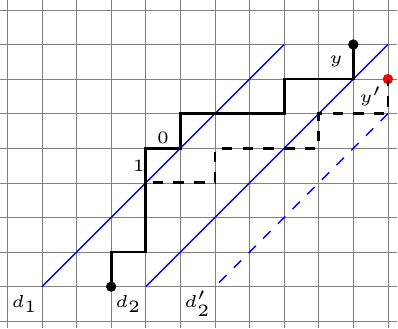}
		\caption{The $\phi$ map. In the figure, $d_1 = -2$ and $d_2=1$.
		The path $L$ is shown as a solid line and $\phi(L)$ is dashed.
		We have the new endpoint $y' = y - \nwvec$ and the shifted diagonal $d'_2=3$.
		Note that $L$ touches $d_2$ in the same manner as $\phi(L)$ touches $d'_2$.
		}\label{fig:phiMap}
	\end{figure}

	It is evident that $\phi$ is invertible. The inverse is given by
	replacing the east step $0$ following the last deepest point (that is, maximizing $y-x$)
  after touching the diagonal $d_1+1$ and before the diagonal $d_2 + 2$ with a north step $1$.

\textbf{Case $d_1 > 0$:}
	In this case, construct a bijection $\psi$ by replacing the east-step $0$ preceding
  $(r, s)$ with a north-step $1$.

	Note that this does not affect the major index and that $\psi(L)$ ends at $\yvec+\nwvec$.
	Furthermore, $\psi(L)$ intersects all diagonals $d_1-1$ as well as $d_2-2$, $d_3-2$, and so on.
	As before, it is straightforward to show that $\psi$ has an inverse.
	This proves the second case of \eqref{eq:diagonalHitRecursion}.
\end{proof}

\begin{corollary}\label{cor:diagonalHitRecursion2}
If $(\yvec; d_1,\dotsc,d_\ell)$ is alternating, then
$H(\yvec; d_1,\dotsc,d_\ell)$ is equal to
\begin{align}\label{eq:diagonalHitRecursionRepeated}
H(\yvec + d_1 \nwvec ; d_2-2d_1,d_3-2d_1,\dotsc,d_\ell-2d_1) \times
\begin{cases}
q^{-d_1} \text{ if $d_1<0$ }\\
 1 \text{ otherwise}.
\end{cases}
\end{align}
\end{corollary}
\begin{proof}
 Apply \cref{lem:diagonalHitMajorIndexCount} repeatedly.
\end{proof}

We shall now focus on generating functions of lattice paths that touch two
diagonals at least $\ell$ times in an alternating fashion.
Let $\langle a, b \rangle_\ell$ denote the alternating list $(a,b,a,b,\dotsc)$ of length $\ell$.

\begin{corollary}\label{cor:evenBouncesRecursion}
Suppose $\delta > \gamma >0$. Then for all $j=0,\dotsc,\lfloor \ell/2 \rfloor $, we have the identities
\begin{equation*}
 H(\yvec;\langle \gamma,\gamma - \delta \rangle_{\ell}) =
 q^{j^2\delta + j\gamma} H(\yvec - j\delta \nwvec ;\langle \gamma + 2j\delta, \gamma - \delta + 2j\delta \rangle_{\ell-2j})
\end{equation*}
and
\begin{equation*}
 H(\yvec;\langle \gamma - \delta, \gamma \rangle_{\ell}) =
 q^{j^2\delta - j\gamma} H(\yvec + j\delta \nwvec ;\langle \gamma-\delta - 2j\delta, \gamma - 2j\delta \rangle_{\ell-2j})
\end{equation*}
\end{corollary}
\begin{proof}
The first identity is proved by applying the recursion in \eqref{eq:diagonalHitRecursionRepeated} two times,
and using induction over $j$.
The first application on the expression
\[
H(\yvec - j\delta \nwvec ;\langle \gamma + 2j\delta, \gamma - \delta + 2j\delta \rangle_{\ell-2j})
\]
gives
\[
H(\yvec -(j\delta+\gamma+2j\delta)\nwvec;\langle -\gamma - \delta - 2j\delta, -\gamma - 2j\delta  \rangle_{\ell-2j-1}).
\]
The second application of the recursion gives
\[
 = q^{\gamma+\delta + 2j\delta}
H(\yvec-(j\delta+\delta)\nwvec;\langle \gamma + 2\delta + 2j\delta, \gamma + \delta + 2j\delta  \rangle_{\ell-2j-2}).
\]
Finally, we observe that
\[
(j^2\delta + j\gamma) + (\gamma+\delta + 2j\delta) = (j+1)^2\delta + (j+1)\gamma,
\]
so the result now follows via induction over $j$.
The second identity is proved in a similar fashion.
\end{proof}

\begin{lemma}\label{lem:qHitCountMajorIndex}
Suppose $\delta > \gamma > 0$ and $\yvec=(n,n-1)$. We then have the identities
\begin{align}
H(\yvec; \langle \gamma, \gamma - \delta \rangle_{2\ell})  &= q^{\ell^2\delta + \ell \gamma} \qbinom{2n-1}{n-1+\delta \ell}_q, \qquad \ell\geq 0 \label{eq:evenLeftFirst}\\
H(\yvec; \langle \gamma, \gamma - \delta \rangle_{2\ell+1})&= q^{\ell^2\delta + \ell \gamma} \qbinom{2n-1}{n-1+\gamma+\delta \ell}_q, \qquad \ell\geq 0, \label{eq:oddLeftFirst} \\
H(\yvec; \langle \gamma-\delta, \gamma \rangle_{2\ell})  &= q^{\ell^2 \delta - \ell \gamma}  \qbinom{2n-1}{n-1-\delta \ell}_q, \qquad \ell\geq 0, \\
H(\yvec; \langle \gamma-\delta, \gamma \rangle_{2\ell-1})&= q^{\ell^2 \delta - \ell \gamma}  \qbinom{2n-1}{n-1+\gamma-\delta \ell}_q, \qquad \ell\geq 1.
\end{align}
\end{lemma}
\begin{proof}
Note that in all cases, we deal with north-east lattice paths of length $2n-1$ with exactly $n$ east-steps,
which we interpret as binary words of length $2n-1$ with exactly $n$ zeros.
From this observation, it is straightforward to see
that \cref{cor:evenBouncesRecursion} together with \cref{lem:easyWordQAnalogue}
implies the first and third identity.

To prove \eqref{eq:oddLeftFirst}, note that \cref{cor:evenBouncesRecursion} and \eqref{eq:diagonalHitRecursion} gives
\begin{align*}
H(\yvec; \langle \gamma, \gamma - \delta \rangle_{2\ell+1} )
         &= q^{\ell^2\delta + \ell \gamma} H(\yvec-(\ell\delta)\nwvec; \langle \gamma + 2\ell \delta \rangle_{1} ) \\
         &= q^{\ell^2\delta + \ell \gamma} H(\yvec+(\gamma+\ell\delta)\nwvec; \langle \cdot \rangle_{0} ) \\
         &= q^{\ell^2\delta + \ell \gamma} \qbinom{2n-1}{(n-1)-(\gamma+\ell\delta)}_q.
\end{align*}

Finally, the last identity follows from the fact that
\begin{align*}
&H(\yvec; \langle \gamma - \delta, \gamma \rangle_{2\ell -1} ) \\
&= q^{(\ell-1)^2\delta - \gamma(\ell-1)} H(\yvec+(-\delta+\ell\delta)\nwvec; \langle \gamma - \delta - 2(\ell-1)\delta \rangle_{1} ) \\
&= q^{(\ell-1)^2\delta - \gamma(\ell-1) -\gamma +\delta + 2(\ell-1)\delta}
H(\yvec+(-\delta+\ell\delta+\gamma -\delta - 2\ell \delta +2\delta)\nwvec; \langle \cdot \rangle_{0})  \\
&= q^{\ell^2\delta - \gamma \ell} H(\yvec+(\gamma-\ell \delta)\nwvec; \langle \cdot \rangle_{0} ) \\
&= q^{\ell^2\delta - \gamma \ell} \qbinom{2n-1}{n-1+\gamma-\ell\delta}_q.
\end{align*}
This finishes the proof of the identities.
\end{proof}

\begin{definition}\label{def:latticeEnumeration}
Fix an integer $w\geq 0$ and $j\in \{ 1,\dotsc,w \}$ and let the diagonals through $(0,0)$ and $(w+2,0)$
be referred to as the \emph{left} and the \emph{right} diagonal, respectively.

Let $L_j(n,w,\ell)$ be the $q$-enumeration of north-east paths $L$ from
\begin{align}\label{eq:startEndPts}
 (w+1-j,0) \text{ to } (n+w+1-j,n-1)
\end{align}
with the property that there are $\ell$ points on the path $L$, $p_1,p_2,\dotsc,p_\ell$, appearing in this order from the start,
such that the odd-indexed $p_i$ lie on the left diagonal, and the even-indexed $p_i$ are points on the right diagonal.
Similarly, define $R_j(n,w,\ell)$ be the be the $q$-enumeration of north-east paths in \eqref{eq:startEndPts},
such that there are $\ell$ points $p_1,p_2,\dotsc,p_\ell$ on the path
with the \emph{even}-indexed $p_i$ being on the left diagonal,
and the \emph{odd}-indexed $p_i$ being points on the right diagonal.
\end{definition}
Let $\delta \coloneqq w+2$ and $\yvec=(n,n-1)$.
From \cref{def:latticeEnumeration}, it is straightforward to see that the
generating functions $L_j(n,w,\ell)$ and $R_j(n,w,\ell)$ are equal
to generating functions in \cref{lem:qHitCountMajorIndex}.
Unraveling the definitions, we have that
\begin{align*}
 L_j(n,w,2\ell)   &= H(\yvec; \langle j+1-\delta, j+1 \rangle_{2\ell})   &= q^{\ell^2 \delta - \ell(j+1)} \qbinom{2n-1}{n-1-\delta\ell}_q \\
 L_j(n,w,2\ell+1) &= H(\yvec; \langle j+1-\delta, j+1 \rangle_{2\ell+1}) &= q^{\ell^2 \delta - \ell(j+1)} \qbinom{2n-1}{n+j-\delta\ell}_q \\
 R_j(n,w,2\ell)   &= H(\yvec; \langle j+1, j+1-\delta \rangle_{2\ell})   &= q^{\ell^2 \delta + \ell(j+1)} \qbinom{2n-1}{n-1+\delta\ell}_q \\
 R_j(n,w,2\ell-1) &= H(\yvec; \langle j+1, j+1-\delta \rangle_{2\ell-1}) &= q^{\ell^2 \delta + \ell(j+1)} \qbinom{2n-1}{n+j+\delta\ell}_q\!\!.
\end{align*}

\begin{corollary}\label{prop:circularQCount}
We have the $q$-enumeration
\begin{align}\label{eq:circularQCountGeneral}
|\CDP(n,w)|_q =
\sum_{s \in \setZ}
\sum_{j=1}^w
q^{s^2\delta + s(j+1)}
\left(
\qbinom{2n-1}{n-1-\delta s}_q
-
\qbinom{2n-1}{n+j+\delta s}_q
\right),
\end{align}
where $\delta = w+2$. In particular, when $w \geq n$, we have
\begin{align}\label{eq:circularQCount2}
 |\CDP(n,w)|_q = w \qbinom{2n-1}{n-1}_q
- \sum_{j=1}^{w} q^{j} \qbinom{2n-1}{n+j}_q
- \sum_{j=1}^{w} \qbinom{2n-1}{n+j-(w+2)}_q.
\end{align}
\end{corollary}
\begin{proof}
We have that $\CDP(n,w)$ are certain north-east lattice paths avoiding the two diagonals through $(0,0)$ and $(\delta,0)$.
To find the maj-count of these paths, we use an inclusion-exclusion argument.
Not taking the restrictions imposed by the diagonals into account, the maj-count is given by
\[
 \sum_{j=1}^w \qbinom{2n-1}{n-1}_q = w \qbinom{2n-1}{n-1}_q.
\]
The paths counted by $L_j(n,w,1)$ and $R_j(n,w,1)$ for $j\in [w]$ enumerate all forbidden paths.
However, we cannot simply subtract both these as there are paths counted by both these expressions,
namely $L_j(n,w,2)$ and $R_j(n,w,2)$, and so on.
Combining \cref{def:latticeEnumeration} and the enumeration in \cref{lem:qHitCountMajorIndex}
then gives the expression in \cref{eq:circularQCountGeneral}.

Note that in particular, $L_j(n,w,2) = R_j(n,w,2) = 0$ whenever $w\geq 0$
(since a path cannot hit both forbidden diagonals in this case) so
we get the less complicated expression in \eqref{eq:circularQCount2}.
Letting $w=n$ in \eqref{eq:circularQCount2} gives \eqref{eq:circularQCount}.
\end{proof}

\begin{lemma}\label{lem:rewittenFormulaForCounting}
We have the identity
\begin{align}\label{lem:enumerationRewritten}
|\CDP(n,w)| &=
(w+2)
\sum_{t \in \setZ}
\binom{2n-1}{n+(w+2)t}
-
\sum_{t \in \setZ}
\binom{2n-1}{n+t} \\
&= (w+2)
\sum_{t \in \setZ}
\binom{2n-1}{n+(w+2)t}
-
2^{2n-1}.
\end{align}
\end{lemma}
\begin{proof}
This is straightforward consequence of \eqref{eq:circularQCountGeneral}, by letting $q=1$
and then adding and subtracting the case $j=0$ and $j=w+1$ to the inner sum.
\end{proof}

We note that the $q=1$ case of \eqref{eq:circularQCountGeneral}
follows easily from \cite[Thm. 2]{Mohanty1979},
where the proof is also done via a reflection argument together with inclusion-exclusion.
However, his approach is not compatible with our use of major index.

\section{The cyclic sieving phenomenon under shifting}

This section contains the proof of our main result, \cref{thm:mainCSP} stated below.

\begin{theorem}\label{thm:mainCSP}
Let $\alpha$ act on $\CDP(n,w)$ by cyclically shifting the area sequence one step.
Then the triple
\[
 \left( \CDP(n,w), \langle \alpha \rangle, |\CDP(n,w)|_q \right)
\]
is a CSP-triple.
\end{theorem}

The proof consists of first counting the number of fixed points under cyclic shift by $k$ steps, which is done in
\cref{lem:fixedPoints}. Then, in \cref{prop:evaluation}, we show that the $q$-analogue $|\CDP(n,w)|_q$ evaluates
to it at $q = e^{2\pi i k/n}$.

Let $\CDP_k(n,w)$ be the subset of area sequences in $\CDP(n,w)$ that is fixed by a cyclic shift of $k$ steps.

\begin{lemma}\label{lem:fixedPoints}
For $n\geq k\geq 1$, let $d \coloneqq \gcd(n,k)$, then
\begin{equation}\label{eq:formulaFixedPointCount}
|\CDP_k(n,w)| = |\CDP(d,w)|.
\end{equation}
\end{lemma}
\begin{proof}
This is easy to prove.
\end{proof}

Let $\delta \coloneqq w+2$. We have that $|\CDP(n,w)|_q$ is equal to
\[
\sum_{s \in \setZ}
\sum_{j=1}^w
q^{s^2\delta + s(j+1)}
\left(
\qbinom{2n-1}{n+\delta s}_q
-
\qbinom{2n-1}{n+j+\delta s}_q
\right).
\]
We need to evaluate this at powers of $\exp(2\pi i/n)$.
All such powers are of the form $\exp(2\pi i \ell/m)$, where $m|n$ and $\gcd(\ell,m)=1$.
The goal is to show that if $n=md$, then
\[
 |\CDP(n,w)|_{q=\exp(2\pi i \ell/m)} = |\CDP(d,w)|.
\]
This identity is trivial whenever $m=1$, so we assume that $m\geq 2$.
Let $\modpart{m}{d}$ denote the (non-negative) remainder of $d$ when divided by $m$.
The $q$-Lucas theorem implies that whenever $q = \exp(2\pi i \ell/m)$ for $n=md$, $\gcd(\ell,m)=1$,
we have that $|\CDP(n,w)|_q$ is equal to
\begin{equation}\label{eq:qLucasVersion}
\sum_{s \in \setZ}
\sum_{j=1}^w
q^{s^2\delta + s(j+1)}
\left(
\binom{2d-1}{d+\lfloor \frac{\delta s}{ m } \rfloor}
\qbinom{m-1}{ \modpart{m}{\delta s}}_q
-
\binom{2d-1}{d+\lfloor \frac{\delta s + j}{ d } \rfloor}
\qbinom{m-1}{ \modpart{m}{j+\delta s}}_q
\right).
\end{equation}

Introduce
\begin{align}
A(s,j) &=\phantom{-} q^{s^2\delta + s(j+1)}
\binom{2d-1}{d+\lfloor \frac{\delta s }{m} \rfloor}
\qbinom{m-1}{ \modpart{m}{\delta s } }_q \\
B(s,j) &= -q^{s^2\delta + s(j+1)}
\binom{2d-1}{d+\lfloor \frac{\delta s + j }{m} \rfloor}
\qbinom{m-1}{ \modpart{m}{ j+\delta s}}_q
\end{align}
which also implicitly depend on $\delta$ and $m$.

The following lemma is needed for the proof of \cref{prop:evaluation}.

\begin{lemma}\label{lem:qCases}
We have the following identities:
\begin{enumerate}[label=(\Roman*)]
\item \label{c:AA} $A(s,j) = -A(-s,w-j),$ for $j \neq w, \delta s \not \equiv_m 0.$

\item \label{c:BB} $B(s,j) = -B(-s-1,w-j),$ for $j \neq w, j + 1 + \delta s \not \equiv_m 0.$

\item \label{c:AB} $A(s,w) = -B(s-1,w)$ for $\delta s \not\equiv_m 0, 1$.

\item \label{c:Sum}
$B(s-1,w) + \sum_{j=1}^w A(-s,j) = -\binom{2d-1}{d +  \frac{\delta s}{m}-1 }$
whenever $\delta s \equiv_m 0$ and $s \not\equiv_m 0$.

\item \label{c:Aj}
$A(s,j) = \binom{2d-1}{d +  \frac{\delta s }{m} }$ whenever $s \equiv_m 0$.

\item \label{c:Aw}
$A(s,w) = -\binom{2d-1}{d +  \frac{\delta s-1}{m}  } $ whenever $\delta s \equiv_m 1$.

\item \label{c:Bj}
$B(s,j) = - \binom{2d-1}{d + \frac{\delta s + j +1}{m} - 1 } $ whenever $j + 1 + \delta s \equiv_m 0$ and $j<w$.

\item \label{c:Bw0}
$B(s-1,w) = \binom{2d-1}{d + \frac{\delta s}{m} -1 } $ whenever $s \equiv_m 0$.

\item \label{c:Bw1}
$B(s-1,w) = -\binom{2d-1}{d + \frac{\delta s -1}{m} - 1}$ whenever $\delta s \equiv_m 1$.
\end{enumerate}

Furthermore, over all combinations of $s\in \setZ$ and $j=1,2,\dotsc,w$,
the above cases covers each term in \cref{eq:qLucasVersion} exactly once.
\end{lemma}
\begin{proof}
\textbf{Case I:}
For $q= \exp(2\pi i \ell/m)$, we want to show $A(s,j)  = -A(-s,w-j)$,
whenever $1\leq j < w$ and $\delta s \not \equiv_m 0$.
We must prove that
 \begin{align*}
 q^{s^2\delta + s(j+1)}
\binom{2d-1}{d+\lfloor \delta s / m \rfloor}
\qbinom{m-1}{  \modpart{m}{\delta s} }_q
&=
-q^{s^2\delta - s(w-j+1)}
\binom{2d-1}{d+\lfloor -\delta s / m \rfloor}
\qbinom{m-1}{ \modpart{m}{-\delta s} }_q.
\end{align*}
Let us first assume that $s>0$ and
let $r \coloneqq (\delta s \mod m)$, so that $0<r<m$.
We must show that
 \begin{align*}
q^{s\delta}
\binom{2d-1}{d+\lfloor \delta s / m \rfloor}
\qbinom{m-1}{ r}_q &=
-\binom{2d-1}{d-1-\lfloor \delta s / m \rfloor}
\qbinom{m-1}{ m-r }_q
\\
q^{r}
\binom{2d-1}{d+\lfloor \delta s / m \rfloor}
\frac{1-q^{m-r}}{1-q^{m}}
\qbinom{m}{r}_q &=
-\binom{2d-1}{d+\lfloor \delta s / m \rfloor}
\frac{1-q^{m-(m-r)}}{1-q^{m}}
\qbinom{m}{ m-r }_q
\\
q^{r}
(1-q^{m-r})
&=
q^{r}-1
\end{align*}
which is true.
The case $s<0$ is treated in a similar manner.

\textbf{Case II:}
For $q= \exp(2\pi i \ell/m)$, we want to show $B(s,j)  = -B(-s-1,w-j)$,
whenever $1\leq j < w$ and $j + \delta s + 1 \not \equiv_m 0$.
We must prove that
 \begin{align*}
 &q^{s^2\delta + s(j+1)}
\binom{2d-1}{d+\lfloor \frac{j + \delta s}{m} \rfloor}
\qbinom{m-1}{  \modpart{m}{\delta s} }_q \\
=
&-q^{(-s-1)^2\delta + (-s-1)(w-j+1)}
\binom{2d-1}{d+\lfloor \frac{w-j+\delta(-s-1)}{m} \rfloor}
\qbinom{m-1}{ \modpart{m}{w-j+\delta(-s-1)} }_q \\
=
&-q^{s(j + \delta s + 1) + j + \delta s + 1}
\binom{2d-1}{d+\lfloor \frac{-2-j-\delta s}{m} \rfloor}
\qbinom{m-1}{ \modpart{m}{-2-j-\delta s} }_q \\
=
&-q^{s(j + \delta s + 1) + j + \delta s + 1}
\binom{2d-1}{d - 1 -\lfloor \frac{j+\delta s}{m} \rfloor}
\qbinom{m-1}{ \modpart{m}{-2-j-\delta s} }_q.
\end{align*}
Now, let $j + \delta s \equiv_m r$, $0 < r < m$. Then, we need to show
 \begin{align*}
 &q^{s(r+1)}
\binom{2d-1}{d+\lfloor \frac{j + \delta s}{m} \rfloor}
\qbinom{m-1}{r}_q \\
=
&-q^{s(r+1) + r + 1}
\binom{2d-1}{d - 1 -\lfloor \frac{j+\delta s}{m} \rfloor}
\qbinom{m-1}{r + 1}_q.
\end{align*}
This follows from that
\[
q^{r + 1} \qbinom{m-1}{r + 1}_q = q^{r + 1}\qbinom{m-1}{r}_q \frac{1-q^{m-r-1}}{1-q^{r+1}} =
q^{r + 1}\qbinom{m-1}{r}_q \frac{1}{q^{r+1}}\frac{q^{r+1}-1}{1-q^{r+1}}
= -\qbinom{m-1}{r}_q.
\]

\textbf{Case III:}
Let $r \equiv_m \delta (s+1)$ with $0\leq r < m$.
We want to prove that $A(s+1, w) = -B(s, w)$
under the condition that $r \notin \{0,1\}$, which implies that $s \neq -1$.

This amounts to proving
 \begin{align*}
 q^{(s+1)^2\delta + (s+1)(\delta-1)}
\binom{2d-1}{d+\lfloor \delta (s+1) / m \rfloor}
\qbinom{m-1}{  \modpart{m}{\delta (s+1)} }_q
= \\
q^{s^2\delta + s(\delta-1)}
\binom{2d-1}{d+\lfloor (\delta (s+1)-2) / m \rfloor}
\qbinom{m-1}{  \modpart{m}{\delta (s+1)-2 } }_q
\end{align*}
Since the binomials are equal under our conditions, it is enough to show that
 \begin{align*}
q^{2 r}
\qbinom{m-1}{ r }_q
&=
q
\qbinom{m-1}{ r-2  }_q
\\
q^{2 r}
\frac{[m-r]_q}{[r]_q}
\qbinom{m-1}{ r-1 }_q
&=
q
\frac{[r-1]_q}{[m-r+1]_q}
\qbinom{m-1}{ r-1 }_q
\\
q^{2 r}
\frac{1-q^{-r}}{1-q^r}
&=
q
\frac{1-q^{r-1}}{1-q^{1-r}}
\end{align*}
and it is easy to verify that these are equal.

\textbf{Case IV:}
We need to prove that
\[
B(s-1, w)  + \sum_{1 \leq j \leq w} A(-s, j) = -\binom{2d-1}{d+\frac{\delta s}{m}-1},
\]
whenever $\delta s \equiv_m 0$, $s \not \equiv_m 0$ and $q=\exp(2\pi i \ell/m)$.
Note that $m\geq 2$ in this case.
Inserting the definitions, we need to evaluate
\begin{align*}
&-q^{(s-1)^2\delta + (s-1)(w+1)}
\binom{2d-1}{d+\lfloor (w+\delta (s-1)) / m \rfloor}
\qbinom{m-1}{ \modpart{m}{ w+\delta (s-1)}}_q
+\\
&\left(\sum_{j=1}^w q^{s^2\delta - s(j+1)}\right)
\binom{2d-1}{d+\lfloor -\delta s / m \rfloor}
\qbinom{m-1}{ \modpart{m}{-\delta s } }_q
\end{align*}
Some simplification gives that this is equal to
\begin{align*}
-q^{1-s}
\binom{2d-1}{d+\lfloor (\delta s-2) / m \rfloor}
\qbinom{m-1}{ \modpart{m}{ \delta s-2}}_q
+
q^{-s}\left(\sum_{j=1}^w q^{-sj}\right)
\binom{2d-1}{d + \frac{\delta s}{m} -1 }
\end{align*}
which becomes
\begin{align*}
-q^{1-s}
\binom{2d-1}{d+ \frac{\delta s}{m} - 1 }
\qbinom{m-1}{ m - 2}_q
+
q^{-s}\left(\sum_{j=1}^w q^{-sj}\right)
\binom{2d-1}{d + \frac{\delta s}{m} -1 }.
\end{align*}
Thus, it suffices to verify that
\[
\left(\sum_{j=1}^{\delta-2} q^{-sj}\right)-q\qbinom{m-1}{ m - 2}_q = -q^{s},
\]
which is straightforward.

\bigskip

\textbf{Case V--IX:} These are straightforward to prove.

%
%
%

\end{proof}

\begin{proposition}\label{prop:evaluation}
Whenever $md = n$, $m\geq 2$ and $\gcd(\ell,m)=1$, we have that
$|\CDP(n,w)|_q$ evaluated at $q=\exp(2\pi i \ell/m)$ is equal to $|\CDP(d,w)|$.
\end{proposition}
\begin{proof}
In \cref{lem:qCases}, the first three cases cancel, so we know that
$|\CDP(n,w)|_q$ evaluated at the root of unity is equal to the sum of the six remaining cases.
After reordering, the sum of the cases is given by the expression
\begin{align}\label{eq:sixSums}
& w\sum_{\substack{s\in \setZ \\  s \equiv_m 0}} \binom{2d-1}{d +  \frac{\delta s }{m} }
-\sum_{\substack{s\in \setZ \\ \delta s \equiv_m 0 \\ s \not\equiv_m 0}} \binom{2d-1}{d + \frac{\delta s}{m}-1 }
+\sum_{\substack{s\in \setZ \\  s \equiv_m 0 }} \binom{2d-1}{d +  \frac{\delta s}{m} -1 } \\
&
-\sum_{\substack{s\in \setZ \\  \delta s \equiv_m 1 }}\binom{2d-1}{d + \frac{\delta s -1}{m} - 1}
-\sum_{j=1}^{w-1} \sum_{\substack{s\in \setZ \\ j + 1 + \delta s \equiv_m 0}}  \binom{2d-1}{d + \frac{\delta s + j +1}{m} - 1 }
-\sum_{\substack{s\in \setZ \\ \delta s \equiv_m 1 }} \binom{2d-1}{d +  \frac{\delta s-1}{m}  } \notag
.
\end{align}
We note that
\[
 \sum_{\substack{s\in \setZ \\ \delta s \equiv_m 1 }} \binom{2d-1}{d +  \frac{\delta s-1}{m}  }
  = \sum_{\substack{s\in \setZ \\ \delta s \equiv_m 1 }} \binom{2d-1}{d + \frac{\delta(-s)+1}{m} - 1 }
   = \sum_{\substack{s\in \setZ \\ \delta s + 1 \equiv_m 0 }} \binom{2d-1}{d + \frac{\delta s+1}{m} - 1 }
\]
so we can merge the fifth and sixth sum in \eqref{eq:sixSums},
and shift the index in the fourth sum.
Furthermore, the second and third sum can be rewritten, by adding and subtracting the third sum.
We get
\begin{align}\label{eq:fiveSums}
& w \sum_{\substack{s\in \setZ \\  s \equiv_m 0}} \binom{2d-1}{d +  \frac{\delta s }{m} }
+2\sum_{\substack{s\in \setZ \\  s \equiv_m 0 }} \binom{2d-1}{d +  \frac{\delta s}{m} -1 }
-\sum_{\substack{s\in \setZ \\ \delta s \equiv_m 0}} \binom{2d-1}{d + \frac{\delta s}{m}-1 }
\\
&
-\sum_{\substack{s\in \setZ \\ \delta - 1 + \delta s \equiv_m 0 }} \binom{2d-1}{d +  \frac{\delta (s+1)-1}{m}  }
-\sum_{j=0}^{w-1} \sum_{\substack{s\in \setZ \\ j + 1 + \delta s \equiv_m 0}}  \binom{2d-1}{d + \frac{\delta s + j +1}{m} - 1 }
\notag
.
\end{align}
The last three terms can be merged, and we do some further simplifications:
\begin{align}\label{eq:fourSums}
& w \sum_{\substack{t \in \setZ}} \binom{2d-1}{d +  \delta t }
+2\sum_{\substack{t \in \setZ  }} \binom{2d-1}{d +  \delta t }
-\sum_{j=0}^{w+1} \sum_{\substack{s\in \setZ \\ j + 1 + \delta s \equiv_m 0}}  \binom{2d-1}{d + \frac{\delta s + j +1}{m} - 1 }
.
\end{align}
Finally, we note that
\[
 \sum_{j=0}^{w+1} \sum_{\substack{s\in \setZ \\ j + 1 + \delta s \equiv_m 0}}  \binom{2d-1}{d + \frac{\delta s + j +1}{m} - 1 }
 =\sum_{r \in \setZ} \binom{2d-1}{d+r}
\]
and thus we have equality with \eqref{lem:enumerationRewritten} in \cref{lem:rewittenFormulaForCounting}.
\end{proof}

\section{The subset cyclic sieving phenomenon}\label{sec:subset}

Recall that $\CDP(n,w)$ is a family of lattice
paths $L$ of length $2n$, ending with a north step.
Note that such a lattice path $L$ is in $\CDP_k(n,w)$ if and only if the binary
word of $L$ is invariant under cyclic shift by $2k$ steps.
However, the action of shifting the area sequence $k$ steps and
shifting the underlying binary word $2k$ steps are not equivalent --- the
family $\CDP(n,w)$ is not closed under such a shifting of the binary word.
This curious observation leads us to make the following definition.

\begin{definition}
Let $Y \subseteq X$ be a set of combinatorial objects and $C_n = \langle g \rangle$
be a cyclic group acting on $X$.
Let $f(q) \in \setZ[q]$ with non-negative coefficients, such that $f(1)=|Y|$.
Then $(Y \subset X ,C_n,f(q))$ is a \emph{subset cyclic sieving phenomenon} if
for every $k\in [n]$ we have
\[
 f(\omega_n^k) = |\{ y\in Y : g^k \cdot y = y \}|.
\]
\end{definition}

We shall need the following theorem from \cite[Thm. 2.7]{AlexanderssonAmini2018}.
\begin{theorem}\label{thm:AA}
Let $f(q)\in \setZ[q]$ take non-negative integer values at all $n$th roots of unity,
and let $X$ be a set of cardinality $f(1)$.
Define
\[
S_k \coloneqq \sum_{j|k} \mu(k/j)f(\omega_n^j) \text{ whenever $k|n$}.
\]
If $S_k\geq 0$ for all $k$, then there is a
cyclic group action $C_n$ acting on $X$, such that $(X,C_n,f(q))$ is a CSP-triple.
\end{theorem}
Note that the integers $S_k$ are exactly the number of elements in $X$ which are in a $C_n$-orbit of size $k$.

\begin{proposition}\label{prop:subsetToCSP}
If $(Y \subset X ,C_n,f(q))$ is a subset-CSP,
then there is a group action $\hat{C}_n$ on $Y$ such that $(Y,\hat{C}_n,f(q))$
is a CSP-triple.
\end{proposition}
\begin{proof}
Let
\[
 S_k \coloneqq \sum_{j|k} \mu(k/j)f(\omega_n^j).
\]
Then $S_k$ is the number of elements in $X$ in a $C_n$-orbit of size $k$,
so $S_k \geq 0$ for all $k\geq 1$.
The fact that a group action $\hat{C}_n$ exists on $Y$ now follows from \cref{thm:AA}.
\end{proof}

\subsection{Lattice paths with subset CSP}

We shall now provide an instance of a subset CSP,
on a family of lattice paths.

Let $\AVL(n,w)$ be the set of lattice paths from $(0,0)$ to $(n,n)$
that never touch the diagonals $\pm w$.
\begin{proposition}
 The maj-count of $\AVL(n,w)$ is given by
\[
 |\AVL(n,w)|_q = \sum_{s \in \setZ} q^{2s^2 w + s w} \left( \qbinom{2n}{n+2sw}_q - \qbinom{2n}{n+w+2sw}_q \right).
\]
\end{proposition}
\begin{proof}
 The proof is analogous to the inclusion-exclusion argument in \cref{prop:circularQCount}.
\end{proof}

We can now provide an example of a subset-CSP on the set $\AVL(n,w)$.
Notice that $\AVL(n,w)$ is a subset of $\AVL(n,n+1)$ --- the set of \emph{all}
lattice paths from $(0,0)$ to $(n,n)$.
\begin{theorem}\label{thm:avlCSP}
Let $n,w \geq 1$ such that $\gcd(n,w)=1$.
Let $C_n$ act on $\AVL(n,w)$ by letting the generator $\beta$
shift the binary word associated with the path two steps.
Then
 \[
  (\AVL(n,w)\subset \AVL(n,n+1), \langle \beta \rangle, |\AVL(n,w)|_q)
 \]
is a subset-CSP-triple.
\end{theorem}
\begin{proof}
We need to evaluate $|\AVL(n,w)|_q$ at $n$th roots of unity.
Let $q = \exp(2\pi i \ell/m)$, where  $m|n$ and $\gcd(\ell,m)=1$,
so that $q$ is a primitive $m$:th root of unity.
Note that it follows that $\gcd(w,m)=1$ as well and we introduce $d=n/m$.
Our goal is to show that $|\AVL(n,w)|_q$ evaluates to
the number of paths in $\AVL(n,w)$ fixed under a shift of $2d$ steps.
It is clear that such paths are in bijection with $\AVL(d,w)$.
There are two cases to consider.

\medskip
\noindent
\textbf{Case $m$ even.}
Using the $q$-Lucas theorem, \cref{lem:qLucas}, we have that
\begin{align}\label{eq:onePathEvenCase1}
 \sum_{s \in \setZ} q^{2s^2 w + s w} \qbinom{2n}{n+2sw}_q =
 \sum_{s \in \setZ} e^{\frac{2\pi i \ell w s (2s + 1)}{m}} \binom{2d}{d + \lfloor 2sw/m \rfloor}\qbinom{0}{\{2sw\}_m}_q.
\end{align}
Notice that the $q$-binomial is $0$ unless $m$ divides $2s$.
Hence, by letting $t \coloneqq 2s/m$, we can rewrite the sum as
\begin{align*}
  \sum_{t \in \setZ} e^{\pi i \ell w t (tm + 1)} \binom{2d}{d + t w} =
  \sum_{t \in \setZ} e^{\pi i  t } \binom{2d}{d + t w}
  =
  \sum_{t \in \setZ} \binom{2d}{d + 2 t w}
  -
  \sum_{t \in \setZ} \binom{2d}{d + w + 2 t w}
\end{align*}
since $\ell$, $w$ and $tm+1$ are all odd if $m$ is even.
In a similar fashion,
\begin{align}\label{eq:onePathEvenCase2}
 \sum_{s \in \setZ} q^{2s^2 w + s w} \qbinom{2n}{n+w+2sw}_q =
 \sum_{s \in \setZ} e^{\frac{2\pi i \ell w s (2s + 1)}{m}} \binom{2d}{d + \lfloor w(2s+1)/m \rfloor}\qbinom{0}{\{w(2s+1)\}_m}_q,
\end{align}
but the last term is always zero, since $m$ does not divide $w(2s+1)$

\medskip
\noindent
\textbf{Case $m$ odd.}
We consider \eqref{eq:onePathEvenCase1}, and see that the $q$-binomial expression vanish
unless $s$ is it is a multiple of $m$. We let $t = s/m$ and obtain
\[
 \sum_{s \in \setZ}\binom{2d}{d+2tw}.
\]
For the other term in \eqref{eq:onePathEvenCase2}, $2s+1$ must be an odd multiple of $m$ in order
for the $q$-binomial to be non-zero. Thus, $t = (2s+1)/m$ is an integer and
the expression simplifies to
\[
 \sum_{s \in \setZ}\binom{2d}{d+w+2tw}.
\]
In conclusion, in both above cases, $|\AVL(n,w)|_q$ evaluates to the expression we have for $|\AVL(d,w)|$.
\end{proof}

\begin{problem}
Find a \emph{natural} cyclic action $\hat{C_n}$ on $\AVL(n,w)$ that makes
\[
(\AVL(n,w), \hat{C}_n, |\AVL(n,w)|_q)
\]
into a CSP-triple.
\end{problem}

\section{Möbius action on binary words}\label{sec:mobius}

\subsection{A new cyclic sieving on binary words}

Let $\BW(n)$ denote the set of binary words of length $n$,
and define an action $\eta$ on $\BW(n)$ as
\[
\eta : (b_1,b_2,\dotsc, b_n) \mapsto (\hat{b}_{n-1},\hat{b}_{n},b_1,b_2,\dotsc,b_{n-2})
\]
where $\hat{b}_i \coloneqq 1-b_i$. Note that the shift is indeed two steps,
and that $\eta^{\circ n}(\bvec) = \bvec$ for all words $\bvec$ of length $n$.
For example,
\[
 \eta(101) = \hat{0}\hat{1}1 = 101, \qquad \eta(110010) = 011100.
\]
We extend the notation so that $\hat{\bvec} \coloneqq (\hat{b}_1,\dotsc,\hat{b}_n)$.

\begin{lemma}\label{lem:binWordMobiusFixedPoints}
The number of words in $\BW(n)$ fixed under $\eta^{\circ m}$
is given by $2^d$ if $\frac{n}{d}$ is odd and $0$ otherwise, where $d = \gcd(m,n)$.
\end{lemma}
\begin{proof}
Because $\eta$ generates a cyclic group of order $n$,
the number of words fixed by $\eta^{\circ m}$  is the same
as the number of words fixed by $\eta^{\circ d}$.
Therefore it suffices to show that if $n = kd$, then
$\eta^{\circ d}$ fixes $2^d$ elements if $k$ is odd, and $0$ elements otherwise.

Let $\bvec \in \BW(n)$ and partition $\bvec$ into $k$ contiguous blocks of length $d$.
Note that $\eta^{\circ d}$ maps block $i$ onto block $i+2$ (mod $k$).
Suppose now $\bvec$ is fixed under $\eta^{\circ d}$ and consider the following cases.

\textbf{Case $k$ even.} We have that
blocks $1,3,5,\dotsc,k-1$ must all be equal.
However, block $k-1$ and block $1$ must also be \emph{different}, as $\eta^{\circ d}$
not only shift bits $2d$ steps to the right, but also flips all bits that wrap around.
This is impossible, so there cannot be any fixed words in this case.

\textbf{Case $k$ odd.}
A similar argument as above shows that
all odd-indexed blocks are equal, all even-indexed blocks are equal,
and an even block is given by flipping all bits in an odd block.
Hence, the entire word is determined by the first block.
There are $2^d$ such possibilities, as there are $d$ bits in a block.

\end{proof}

\begin{lemma}\label{lem:bwqpoly}
For fixed $n$, all the expressions
\begin{align}\label{eq:binaryWordsQBinomialIdentity}
  (A)\quad \sum_{k=0}^n q^{\binom{k}{2}} \qbinom{n}{k}_q
  &&(B)\quad \prod_{j=0}^{n-1}(1+q^j)
  &&(C)\quad \sum_{\bvec \in \BW(n)}  q^{\maj(\bvec) + \maj(\hat{\bvec})}
\end{align}
are equal.
\end{lemma}
\begin{proof}
\textit{Identity $(A)=(B)$.}
This is simply a consequence of the $q$-binomial theorem, (see \cite[p. 14]{KacCheung2001})
\[
 \prod_{j=0}^{n-1}(1+xq^k) = \sum_{k=0}^n  q^{\binom{k}{2}} \qbinom{n}{k}_q x^k.
\]
\textit{Identity $(B)=(C)$.} We do induction over $n$. The base case $n=1$ is easy.
Now assume that the identity hold for $n-1$. Consider a binary word $\bvec$ of length $n-1$.
We can either append $0$ or $1$ to make a word $\bvec'$ of length $n$.
If the last bit of $\bvec$ is equal to the appended bit,
\[
 \maj(\bvec) + \maj(\hat{\bvec}) = \maj(\bvec') + \maj(\hat{\bvec'}),
\]
otherwise, the right hand side larger by $n$.
\end{proof}

\begin{proposition}\label{prop:mobiusCSPonWords}
Let $\eta$ act on the binary words $\BW(n)$ as before.
Then
\[
\left(\BW(n), \langle \eta \rangle, \sum_{k=0}^n q^{\binom{k}{2}} \qbinom{n}{k}_q \right)
\]
is a CSP-triple.
\end{proposition}
\begin{proof}
We need to evaluate $B_n(q)\coloneqq \sum_{k=0}^n q^{\binom{k}{2}} \qbinom{n}{k}_q$
at $n$th roots of unity.
Suppose $n = md$ and $\xi = e^{2\pi i \frac{\ell}{m}}$ with $\gcd(\ell,m)=1$.
Note that \cref{lem:qLucas} gives that
\[
\qbinom{n}{k}_{\xi} =
\begin{cases}
  \binom{d}{k/m} & \text{if } m|k \\
  0 & \text{otherwise.}
\end{cases}
\]
Therefore,
\[
 B_n(\xi) = \sum_{k=0}^n \xi^{\binom{k}{2}} \qbinom{n}{k}_\xi
 =
 \sum_{j=1}^{d} \xi^{\binom{mj}{2}} \qbinom{n}{mj}_\xi
\]
since only terms in the left hand side where $k$ is a multiple of $m$ contribute.
We then get that
\begin{align*}
B_n(\xi) &= \sum_{j=0}^d e^{2\pi i \frac{\ell}{m} \binom{mj}{2}} \binom{d}{j}
        = \sum_{j=0}^d (-1)^{\ell j(mj-1)} \binom{d}{j} =
	\begin{cases}
	0 & \text{if}\ m \ \text{even},  \\
	2^d & \text{otherwise}.
	\end{cases}
\end{align*}
The two cases in the last step is as follows: if $m$ is even,
then $\ell$ must be odd, and it follows that the sum is $0$.
If $m$ is odd, then every term is positive and we get $2^d$.

That $\left(\BW(n), \langle \eta \rangle, \sum_{k=0}^n q^{\binom{k}{2}} \qbinom{n}{k}_q \right)$ exhibits
the CSP now follows from \cref{lem:binWordMobiusFixedPoints}.
\end{proof}

\subsection{Cyclic sieving on circular Möbius paths}

Recall the definition of circular Möbius paths,
and \cref{lem:mobiusBijection}, showing that $|\CMP(n)|=2^{n-1}$.
As $\CMP(n) \subseteq \CDP(n)$, we use the same
definition of major index for circular Möbius paths as for the circular Dyck paths.

Let $\OBW(n) \subset \BW(n)$ be the set of binary words with odd parity.
Note that the last bit can be deduced from the remaining word.
We have a bijection $M: \OBW(n) \to \CMP(n)$ by
\[
 M(b_1,\dotsc,b_n) \mapsto (b_1,b_2,\dotsc,b_{n-1},0, \hat{b_1}, \hat{b_2},\dotsc, \hat{b}_{n-1},1).
\]
Since $\eta$ preserves the parity of the word, $\eta$ act on $\OBW(n)$
and thus induces a $C_n$ action $\tilde{\eta}$ on $\CMP(n)$.

\begin{example}
Consider the first path in \cref{fig:mobiusExample} of length $16$.
The first half of the corresponding binary word
is given by \texttt{10110110}, which is identified with $\texttt{10110110}\in \OBW(n)$.
We apply $\eta$ to this word and get
$\mathtt{01101101}$.
We drop the last bit and append a $1$: \texttt{01101101} this determines a new Möbius path in $\CMP(8)$,
namely second path in \cref{fig:mobiusExample}.
\end{example}

\begin{lemma}\label{lem:mobiusSameQDistr}
We have the following formula for $|\CMP(n)|_q$, mod $(q^n-1)$:
\[
\sum_{\bvec \in \CMP(n) } q^{\maj(\bvec)} \equiv \frac{1}{2}\sum_{k=0}^n q^{\binom{k}{2}} \qbinom{n}{k}_q \mod (q^n-1).
\]
\end{lemma}
\begin{proof}
 First note that by definition in \eqref{eq:mobiusCondition} that
 \begin{equation}\label{eq:cmpqFormula}
  |\CMP(n)|_q = \sum_{\bvec \in \CMP(n) } q^{\maj(\bvec)} =
  \sum_{\substack{\bvec \in \BW(n) \\ \bvec_n = 0 }} q^{\maj(\bvec\sim \hat{\bvec})}
 \end{equation}
where $\sim$ denotes concatenation.
Now observe that $\maj(\bvec\sim \hat{\bvec} ) \equiv_n \maj(\bvec) + \maj(\hat{\bvec})$.
Any descent in the second half of the concatenation contributing to $\maj$ can simply be shifted by $n$.
Furthermore, any contribution to $\maj$ by a descent between the first and
second half must be exactly $0$ or $n$.
It is also clear by symmetry that
\[
 \sum_{\substack{\bvec \in \BW(n) \\ \bvec_n = 0 }} q^{\maj(\bvec\sim \hat{\bvec})} =
 \sum_{\substack{\bvec \in \BW(n) \\ \bvec_n = 1 }} q^{\maj(\bvec\sim \hat{\bvec})}.
\]
This together with \cref{lem:bwqpoly} implies \eqref{eq:cmpqFormula}.
\end{proof}

We are now ready to present a cyclic sieving phenomenon on circular Möbius paths of size $n$.
\begin{theorem}
The triples
\[
\left(\CMP(n), \langle \tilde{\eta} \rangle, \frac{1}{2}\sum_{k=0}^n q^{\binom{k}{2}} \qbinom{n}{k}_q \right)
\text{ and }
\left(\CMP(n), \langle \tilde{\eta} \rangle, |\CMP(n)|_q\right)
\]
are CSP-triples, where $\langle \tilde{\eta} \rangle$
is the cyclic group of order $n$ that acts on $\CMP(n)$.
\end{theorem}
\begin{proof}
First, \cref{lem:mobiusSameQDistr} implies that both CSP-instances
are the same, up to choice of nice polynomial in $q$.
The map $M$ above shows that the cyclic sieving phenomenon on $\CMP(n)$ is simply half
of the cyclic sieving phenomenon in \cref{prop:mobiusCSPonWords}.
\end{proof}

\begin{remark}
There is an alternative way to prove CSP on $\CMP(n)$.
We view elements in $\CMP(n)$ as $(\xvec,\bvec)$, a starting point and a path with $2n$
steps. Let $\beta$ act on the binary word $\bvec$ by
cyclically shifting the path by two steps.
Clearly, $\beta$ does not preserve the set $\CMP(n)$, but it is fairly easy to prove that
\[
 \left(\CMP(n) \subset X, \langle \beta \rangle, |\CMP(n)|_q\right)
\]
is a subset-CSP-triple, where $X$ is chosen appropriately.
\cref{prop:subsetToCSP}, now implies the existence of a CSP
on $\CMP(n)$ with $|\CMP(n)|_q$ as CSP-polynomial.
\end{remark}

\section{Lyndon-like cyclic sieving}\label{sec:lyndon}

Most results on cyclic sieving regards a family
of combinatorial objects, where the cyclic group $C_n$ acts on a set $X_n$.
The main result of this paper is no exception.
In such cases, it is natural to pay extra attention to families where
the various fixed-points in $X_n$ under elements in $C_n$ are
in bijection with $X_k$ for some $k \leq n$.
This occur when the group action is some type of cyclic shift on words,
as we shall see.

\begin{definition}\label{def:shiftLikeCSP}
Let $\{(X_n, C_n, f_n(q))\}_{n=1}^\infty$ be a family of instances of the cyclic sieving phenomenon.
We say that the family is \defin{Lyndon-like} if for every pair of positive integers $m$, $n$, with $m|n$,
we have
\[
 f_{n/m}(1) = f_n\left( e^{\tfrac{2 \pi i}{m}} \right).
\]
By the definition of CSP, we have that
\[
 f_n\left( e^{\tfrac{2 \pi i}{m}} \right) = |\{ x \in X_n : g^{n/m}(x) = x \}|,
\]
where $\langle g \rangle = C_n$. Hence, the family is Lyndon-like if and only if
the number of elements in $X_n$ fixed under $g^{d}$ is equal to $|X_{d}|$, for
every $d|n$.
\end{definition}

Apart from \cref{thm:mainCSP} and \cref{thm:avlCSP}, there are several other Lyndon-like families of CSP.
Here we list a few others.
\begin{enumerate}
 \item Words of length $n$, in the alphabet $[k]$, with
 $f_n(q) = \sum_{w \in [k]^n} q^{\maj(w)}$ as the polynomial.
 \item In \cite{Uhlin2019}, a Lyndon-like CSP instance
 related to non-symmetric Macdonald polynomials is conjectured.
 Here, the family $f_n(q)$ is defined as $f_n(q) = \sum_{T \in NAF(n\lambda,k)} q^{\maj(T)}$,
 where $NAF(\lambda,k)$ is a certain set of
 \defin{non-attacking fillings} with  maximal entry at most $k$ and shape $\lambda$.
 This CSP generalizes the Lyndon-like CSP on words.
\end{enumerate}

\begin{lemma}\label{lem:lyndonParams}
Let $\{(X_n, C_n, f_n(q))\}_{n=1}^\infty$
be a Lyndon-like family of instances of the CSP.
Then there are unique non-negative integers $\{t_d\}_{d=1}^\infty$ such that
for every $n\geq 1$ we have
\[
 |X_n| = \sum_{d|n} d \cdot t_d.
\]
\end{lemma}
\begin{proof}
Let $O_{n,k}$ be the set of elements in $X_n$ that are in an
orbit of size $k$ under $g$.
We let $t_n \coloneqq \frac{1}{n}|O_{n,n}|$, so the identity we wish to prove is
via Möbius inversion equivalent to the right hand side of
\begin{equation}\label{eq:tdEqn}
 |X_n| = \sum_{d|n} d \cdot t_d \quad \Longleftrightarrow \quad |O_{n,n}| = \sum_{d|n} \mu\left(\frac{n}{d}\right) |X_d|.
\end{equation}
Now since the family is Lyndon-like, we have that for all $d|n$,
\begin{equation}\label{eq:tdEqn2}
  |X_d| = |\{ x \in X_n : g^d(x) = x \}| = \sum_{\substack{1\leq k \leq d \\ k|d } } |O_{n,k}|.
\end{equation}
Thus, combining \eqref{eq:tdEqn} and \eqref{eq:tdEqn2}, it suffices to show that
\[
 |X_n| = \sum_{\substack{1\leq d \leq n \\ d|n }}
  \mu\left(\frac{n}{d}\right)
  \sum_{\substack{1\leq k \leq d \\ k|d }}
  |O_{n,k}|.
  \]
Möbius inversion on the outer sum gives that
  \[
  |X_n| =
  \sum_{\substack{1\leq k \leq n \\ k|n }}
  |O_{n,k}|
\]
which is obviously true since every element in $X_n$ belongs to exactly one orbit of some size.
The parameters $t_n$ are by construction unique.
\end{proof}

We let the integers $t_d$ be called \emph{Lyndon parameters},
since if we choose $t_d$ to be the usual Lyndon numbers, we get that $|X_n| = 2^n$.
A CSP instance with these parameters can be constructed by considering binary words of length $n$,
with $C_n$ acting via cyclic shift.
The Lyndon words of length $n$ are then in bijection with $\frac{1}{n}|O_{n,n}|$ in the notation above ---
that is, they are representatives of orbits of size $n$ under cyclic shift.

\begin{lemma}
 There is no Lyndon-like cyclic sieving phenomenon where
 $C_n$ act on on some family of Catalan objects $\Cat(n)$.
 \end{lemma}
 \begin{proof}
  We have that $|\Cat(1)|=1$ and $|\Cat(3)|=5$, so we must have $C_3$ acting on $5$ objects.
  But then, we must have $t_1=2$, $t_3=1$ or $t_1=5$, both of which is incompatible with $|\Cat(1)|=1$.
 \end{proof}

The goal of the remainder of this section is to prove the converse of \cref{lem:lyndonParams}.
That is, there is a Lyndon-like family of CSP instances for
any choice of Lyndon parameters.

\begin{proposition}\label{prop:lyndonCSP}
 For any sequence of Lyndon parameters $T=\{t_d\}_{d=1}^\infty$,
 there is a Lyndon-like family of instances of the CSP.
\end{proposition}
\begin{proof}
It is enough to construct $X_n$ and a $C_n$-action on $X_n$
with the properties in \cref{lem:lyndonParams}.
Let
\[
 X_n = \{(d,i,j) : d | n,\quad 1\leq i \leq t_k, \quad 1\leq j \leq d \}
\]
and let the generator $g$ act as $g \cdot (d,i,j) = (d,i,j+n/d)$, where the last coordinate is
taken modulo $d$. It is straightforward to show from the construction that for all $n$ and $k|n$
\[
 |X_n| = \sum_{d|n} d\cdot t_d \quad \text{ and } \quad  |X_k| = |\{ x \in X_n : g^k(x) = x \}|.
\]
We can then use $f_n(q)$ from \cref{prop:cspPolyFromGroupAction} to make $(X_n,C_n,f_n(q))$
into a CSP-triple.
\end{proof}

It is possible to biject the $C_n$-action in \cref{prop:lyndonCSP}  into
an action on certain (however, quite artificial) words of length $n$, where $C_n$ act by a one-step shift.
Thus, it is no accident that the examples we have given above are all of this form.

Also note that the Lyndon parameters uniquely define the family of instances of the CSP
in the following sense.
\begin{proposition}
 Let $(X_n,C_n,f_n(q))$ and $(Y_n,C_n,g_n(q))$ be two families of Lyndon-like instances of the CSP with
 the same Lyndon parameters. Then there are $C_n$-equivariant bijections $\psi_n : X_n \to Y_n$
 such that for all $g\in C_n$ and $x \in X_n$, $\psi_n(g \cdot x) = g \cdot \psi_n(x)$.
\end{proposition}
\begin{proof}
  First, we have that $|X_n| = |Y_n|$ for all $n$, by \cref{lem:lyndonParams}.
 From \cref{def:shiftLikeCSP}, it then follows that
 \[
  \{x \in X_n : g^{n/k} \cdot x =x \} = |X_k|=|Y_k| = \{y \in Y_n : g^{n/k}\cdot y=y \}.
 \]
By using Möbius inversion on \eqref{eq:tdEqn2},
it is clear that the number of $C_n$-orbits of size $d$ in $X_n$ is
equal to the number of $C_n$-orbits of size $d$ in $Y_n$.
We can then simply let $\psi_n$ map orbits to orbits in an equivariant manner.
\end{proof}

\begin{remark}
 Note that \cref{def:shiftLikeCSP} gives a method to \emph{computationally check}
 if a sequence of polynomials $\{f_n(q)\}_{n=1}^\infty$
 might be completed to a Lyndon-like family of CSP.
 In this case, one might be able to narrow down the search for a suitable group action.
\end{remark}

\begin{problem}
 Given  $T=\{t_d\}_{d=1}^\infty$, find a natural family $\{ f_n(q) \}_{n=1}^\infty$
 so that $f_n(q)$ is the sequence of polynomials in a Lyndon-like CSP family with Lyndon parameters $T$.
\end{problem}

\section{Homomesy under area shift}

There is a concept called homomesy that means that a statistic
has the same average in each orbit as it has in the full space.

\begin{theorem}
For the words from circular Dyck paths, $\inv(w)$ is homomesic with
respect to the action $\alpha$, i.e. rotation by 1.
\end{theorem}

\begin{example} For $n=2$, we have the binary strings ending with a one divided into orbits by $\alpha$:
$\{0011,1001\}, \{0101\}$. The inversion numbers are $\inv(0011)=4$, $\inv(1001)=2$, $\inv(0101)=3$,
so the average for each orbit is 3. Note that the last orbit corresponds to two different CDP,
but the inversion number is the same, so it is not important to keep track of here.
\end{example}

\begin{proof}
We only have to look at the orbit for any given word $w$, that has $n$ zeros, $n$
ones and ends in a one. It turns out that it is here not important that it comes from CDP.
Let $z(w)=(z_1,\dotsc,z_{n})$ be the
number of zeros between two consecutive ones in $w$, i.e. $z_i$ is the number of zeros between the
$(i-1)$th and $i$th one (with $z_1=$number of starting zeros).
For example $w=011001$ has $z(w)=(1,0,2)$.
Then $\inv(w)=\sum_{i=1}^n (n-i+1)z_i$.
The action $\alpha$ is rotating $z$ one step.
Thus summing over an orbit of size $n$ we get
\[
\sum_i\sum_j (n-j+1)z_i=\sum_i z_i\binom{n+1}{2}=n\cdot\binom{n+1}{2}.
\]
Hence the average over the orbit is $\binom{n+1}{2}$.
If the orbit is of size shorter than $n$ it is still divisible by $n$ and the same calculation holds.
\end{proof}

The number of inversions is not homomesic with respect to the other two actions we have discussed;
shift by 2 of a binary word or shift of the area sequence.

\section{Further research directions}

Computer experiments suggests that the
cyclic sieving phenomenon in \cref{thm:mainCSP} can be refined,
by taking the number of valleys of the circular Dyck path into account.
A \emph{valley} of a (circular) Dyck path is an index $i \in [n]$ such that $a_{i+1} \leq a_i$
in its area sequence, where the index is taken mod $n$.
In \cref{ex:circularDyckPath}, $2$, $4$, $6$ are valleys.

Furthermore, one can also consider Schröder paths where diagonal steps are allowed.
We do not have an enumeration formula for these, it is an interesting open problem
as that would count certain circular vertical-strip
LLT polynomials, see \cite{AlexanderssonPanova2016}.

One could also introduce $\CMP(n,w) \subseteq \CDP(n,w)$
as the set of circular area sequences, satisfying the additional Möbius restriction
in \eqref{eq:mobiusCondition}.
Using the machinery above with some modification,
it should be fairly easy to derive analogous expressions for enumeration.

\subsection*{Acknowledgement}

P.A. is funded by the \emph{Knut and Alice Wallenberg Foundation} (2013.03.07).
S.L. and S.P. were supported by the Swedish Research Council grant 621-2014-4780.

\bibliographystyle{alphaurl}
\bibliography{bibliography}

\begin{thebibliography}{RSW04}

\bibitem[AA18]{AlexanderssonAmini2018}
Per Alexandersson and Nima Amini.
\newblock The cone of cyclic sieving phenomena.
\newblock {\em ArXiv e-prints}, 2018.
\newblock \href {http://arxiv.org/abs/1804.01447} {\path{arXiv:1804.01447}}.

\bibitem[And76]{Andrews1976}
G.~E. Andrews.
\newblock {\em The Theory of Partitions. (Encyclopedia of Mathematics and Its
  Applications)}, volume~2.
\newblock Addison-Wesley, Reading, Mass., 1976.

\bibitem[AP18]{AlexanderssonPanova2016}
Per Alexandersson and Greta Panova.
\newblock {LLT} polynomials, chromatic quasisymmetric functions and graphs with
  cycles.
\newblock {\em Discrete Mathematics}, 341(12):3453--3482, dec 2018.
\newblock \href {http://dx.doi.org/10.1016/j.disc.2018.09.001}
  {\path{doi:10.1016/j.disc.2018.09.001}}.

\bibitem[AS18]{AhlbachSwanson2018}
Connor Ahlbach and Joshua~P. Swanson.
\newblock Refined cyclic sieving on words for the major index statistic.
\newblock {\em European Journal of Combinatorics}, 73:37--60, oct 2018.
\newblock \href {http://dx.doi.org/10.1016/j.ejc.2018.05.003}
  {\path{doi:10.1016/j.ejc.2018.05.003}}.

\bibitem[BM12]{Baur2012}
Karin Baur and Volodymyr Mazorchuk.
\newblock {C}ombinatorial analogues of ad-nilpotent ideals for untwisted affine
  {L}ie algebras.
\newblock {\em Journal of Algebra}, 372:85--107, December 2012.

\bibitem[com08]{Sage-Combinat}
The Sage-Combinat community.
\newblock Sage--combinat: enhancing {S}age as a toolbox for computer
  exploration in algebraic combinatorics.
\newblock \url{http://combinat.sagemath.org}, 2008.

\bibitem[Dev19]{sage}
The~Sage Developers.
\newblock Sagemath, the {S}age {M}athematics {S}oftware {S}ystem {V}ersion 8.6.
\newblock \url{http://www.sagemath.org}, 2019.

\bibitem[Ell16]{Ellzey2016}
Brittney Ellzey.
\newblock Chromatic quasisymmetric functions of directed graphs.
\newblock {\em ArXiv e-prints}, 2016.
\newblock \href {http://arxiv.org/abs/1612.04786} {\path{arXiv:1612.04786}}.

\bibitem[FH85]{FurlingerHofbauer1985}
J.~F\"{u}rlinger and J.~Hofbauer.
\newblock q-{C}atalan numbers.
\newblock {\em Journal of Combinatorial Theory, Series A}, 40(2):248--264, nov
  1985.
\newblock \href {http://dx.doi.org/10.1016/0097-3165(85)90089-5}
  {\path{doi:10.1016/0097-3165(85)90089-5}}.

\bibitem[Hag07]{qtCatalanBook}
James Haglund.
\newblock {\em The q,t-{C}atalan {N}umbers and the {S}pace of {D}iagonal
  {H}armonics ({U}niversity {L}ecture {S}eries)}.
\newblock American Mathematical Society, 2007.
\newblock URL: \url{https://www.math.upenn.edu/~jhaglund/books/qtcat.pdf}.

\bibitem[KC01]{KacCheung2001}
Victor Kac and Pokman Cheung.
\newblock {\em Quantum Calculus (Universitext)}.
\newblock Springer, 2001.

\bibitem[Moh79]{Mohanty1979}
Sri~Gopal Mohanty.
\newblock {\em Lattice Path Counting and Applications}.
\newblock Probability and Mathematical Statistics. Academic Press, New York,
  1979.
\newblock ISBN~0-12-504050-4.

\bibitem[Pec14]{Pechenik2014}
Oliver Pechenik.
\newblock Cyclic sieving of increasing tableaux and small {S}chr{\"{o}}der
  paths.
\newblock {\em Journal of Combinatorial Theory, Series A}, 125:357--378, jul
  2014.
\newblock \href {http://dx.doi.org/10.1016/j.jcta.2014.04.002}
  {\path{doi:10.1016/j.jcta.2014.04.002}}.

\bibitem[PPR08]{PetersenPylyavskyyRhoades2008}
T.~Kyle Petersen, Pavlo Pylyavskyy, and Brendon Rhoades.
\newblock Promotion and cyclic sieving via webs.
\newblock {\em Journal of Algebraic Combinatorics}, 30(1):19--41, sep 2008.
\newblock \href {http://dx.doi.org/10.1007/s10801-008-0150-3}
  {\path{doi:10.1007/s10801-008-0150-3}}.

\bibitem[Rho10]{Rhoades2010}
Brendon Rhoades.
\newblock Cyclic sieving, promotion, and representation theory.
\newblock {\em Journal of Combinatorial Theory, Series A}, 117(1):38--76, jan
  2010.
\newblock \href {http://dx.doi.org/10.1016/j.jcta.2009.03.017}
  {\path{doi:10.1016/j.jcta.2009.03.017}}.

\bibitem[RSW04]{ReinerStantonWhite2004}
V.~Reiner, D.~Stanton, and D.~White.
\newblock The cyclic sieving phenomenon.
\newblock {\em Journal of Combinatorial Theory, Series A}, 108(1):17--50, Oct
  2004.
\newblock \href {http://dx.doi.org/10.1016/j.jcta.2004.04.009}
  {\path{doi:10.1016/j.jcta.2004.04.009}}.

\bibitem[Sag92]{Sagan1992}
Bruce~E Sagan.
\newblock Congruence properties of q-analogs.
\newblock {\em Advances in Mathematics}, 95(1):127--143, sep 1992.
\newblock \href {http://dx.doi.org/10.1016/0001-8708(92)90046-n}
  {\path{doi:10.1016/0001-8708(92)90046-n}}.

\bibitem[Sag11]{Sagan2011}
Bruce Sagan.
\newblock The cyclic sieving phenomenon: a survey.
\newblock In Robin Chapman, editor, {\em Surveys in Combinatorics 2011}, pages
  183--234. Cambridge University Press, 2011.
\newblock \href {http://dx.doi.org/10.1017/cbo9781139004114.006}
  {\path{doi:10.1017/cbo9781139004114.006}}.

\bibitem[Slo16]{OEIS}
Neil J.~A. Sloane.
\newblock The on-line encyclopedia of integer sequences, 2016.
\newblock URL: \url{https://oeis.org}.

\bibitem[Sta11]{StanleyEC1}
Richard~P. Stanley.
\newblock {\em Enumerative {C}ombinatorics: {V}olume 1}.
\newblock Cambridge University Press, 2nd edition, 2011.

\bibitem[Sta15]{StanleyCatalan}
Richard~P. Stanley.
\newblock {\em Catalan Numbers}.
\newblock Cambridge University Press, 2015.

\bibitem[Ste94a]{Stembridge1994b}
John~R. Stembridge.
\newblock On minuscule representations, plane partitions and involutions in
  complex {L}ie groups.
\newblock {\em Duke Mathematical Journal}, 73(2):469--490, feb 1994.
\newblock \href {http://dx.doi.org/10.1215/s0012-7094-94-07320-1}
  {\path{doi:10.1215/s0012-7094-94-07320-1}}.

\bibitem[Ste94b]{Stembridge1994}
John~R. Stembridge.
\newblock Some hidden relations involving the ten symmetry classes of plane
  partitions.
\newblock {\em Journal of Combinatorial Theory, Series A}, 68(2):372--409, nov
  1994.
\newblock \href {http://dx.doi.org/10.1016/0097-3165(94)90112-0}
  {\path{doi:10.1016/0097-3165(94)90112-0}}.

\bibitem[Ste96]{Stembridge1996}
John~R. Stembridge.
\newblock Canonical bases and self-evacuating tableaux.
\newblock {\em Duke Mathematical Journal}, 82(3):585--606, mar 1996.
\newblock \href {http://dx.doi.org/10.1215/s0012-7094-96-08224-1}
  {\path{doi:10.1215/s0012-7094-96-08224-1}}.

\bibitem[Uhl19]{Uhlin2019}
Joakim Uhlin.
\newblock Combinatorics of {M}acdonald polynomials and cyclic sieving.
\newblock Master's thesis, KTH, Mathematics (Div.), 2019.
\newblock URL:
  \url{http://kth.diva-portal.org/smash/record.jsf?pid=diva2%3A1282825}.

\end{thebibliography}
\end{document}